\documentclass[reqno,12pt,letterpaper]{amsart}
\usepackage{amsmath,amssymb,amsthm,graphicx,mathrsfs,url}
\usepackage{subfigure} 
\usepackage[usenames,dvipsnames]{color}
\usepackage[colorlinks=true,linkcolor=Red,citecolor=Green]{hyperref}

\def\?[#1]{\textbf{[#1]}\marginpar{\Large{\textbf{??}}}}

\newtheorem{theo}{Theorem}
\newtheorem{prop}{Proposition}[section]

\newtheorem{lem}[prop]{Lemma}
\newtheorem{cor}[prop]{Corollary}

\theoremstyle{remark}
\newtheorem{rem}{Remark}
\numberwithin{equation}{section}

\DeclareMathOperator{\supp}{supp}

\newcommand{\DR}{\frac{d}{dr}}

\newcommand{\DDZ}{\frac{d^2}{dz^2}}

\newcommand{\DOM}{\mathcal{D}}
\newcommand{\TE}{\tilde{E}}

\renewcommand{\Im}{\operatorname{Im}}
\renewcommand{\Re}{\operatorname{Re}}

\title[Quasinormal modes for Schwarzschild--AdS black holes]
{Quasinormal modes for Schwarzschild--AdS black holes:
exponential convergence to the real axis.}
\author{Oran Gannot}
\email{ogannot@math.berkeley.edu}
\address{Department of Mathematics, Evans Hall, University of California,
Berkeley, CA 94720, USA}

\begin{document}

\begin{abstract}
We study quasinormal modes for massive scalar fields in Schwarzschild--anti-de Sitter black holes. When the mass-squared is above the Breitenlohner--Freedman bound we show that for large angular momenta, $\ell$, there exist quasinormal modes with imaginary parts of size $\exp(-\ell/C)$. We provide an asymptotic expansion for the real parts of the modes closest to the real axis and identify the vanishing of certain coefficients depending on the dimension.
\end{abstract}

\maketitle


\section{Introduction}
\label{int}

Quasinormal modes for Schwarzschild--AdS black holes are a subject of active study in current
physics literature -- see \cite{Berti:2009QNM},\cite{Konoplya:2011} and references given there. These modes are mathematically 
defined as poles of the Green function for the stationary problem and are a special 
case of {\em scattering resonances} --- see for example \cite{Zworski:1999}. 

Following established tradition we separate variables after which the inverse of 
the angular momentum, $ \ell$, becomes a semiclassical parameter $ h $. 
In this note we construct approximate solutions (quasimodes) to the stationary equation 
with errors of size $ \exp ( - C/h) $. We then apply a modified version of 
the results of Tang-Zworski \cite{Tang:1998} and Stefanov \cite{Stefanov:2005} to show the existence
of quasinormal modes. (The confusing nomenclature seems unavoidable when
following trends in the literature: quasimodes refer to approximate solutions
and quasinormal modes to the poles of the Green function, namely scattering resonances.) This passage from quasimodes to resonances does not depend on the reduction to one dimension, nor on the analyticity of the potential. Additionally, most of the auxiliary techniques used are suited for higher dimensional analysis. It is likely that a more
refined description of quasinormal modes (especially of the imaginary parts) is possible
using exact WKB methods \cite{Fujiie:2000},\cite{Ramond:1996}, and encouraging progress has been made in the physics
literature \cite{Festuccia:2008},\cite{Dias:2012},\cite{Grain:2006}. 

Quasinormal modes are defined using the meromorphic continuation of the Green function.
The existence of a meromorphic continuation follows from the general ``black box'' formalism in scattering
theory \cite{Sjostrand:1991}, \cite{Sjostrand:1996} using the the method of complex scaling. In a forthcoming paper \cite{Gannot1} we adapt this
formalism to the case of exponentially decaying perturbations of the Laplacian outside a compact set, with no analyticity assumptions. We should stress, however, that in the exact Schwarzschild--AdS setting the complex scaling approach of 
\cite{Sjostrand:1991} is also available. In the analytic black box framework it is also known that the poles of the meromorphic continuation of the resolvent agree with the poles of the scattering matrix \cite{Nedelec:2004}.

The Schwarzschild--anti-de Sitter metric in $d+1$ dimensions is a spherically symmetric solution of the vacuum Einstein equation with negative cosomological constant. Introduce the function
\[
f(r)=r^2+1-\frac{\mu}{r^{d-2}}.
\]
The parameter $\mu$ is a positive constant proportional to the mass of the black hole. Let $r_+$ denote the unique positive root of $f$; this radius define the event horizon. The region outside the horizon is the product $(0,\infty)_{t} \times (r_+,\infty)_{r} \times S^{d-1}$ and in these coordinates the metric takes the form
\begin{equation}
g = -f\,dt^2 + \frac{1}{f}dr^2 + r^2 d\Omega_{d-1}^2,
\end{equation}
where
$d\Omega_{d-1}^2$ is the standard metric on the sphere $S^{d-1}$. We will also make extensive use of the Regge-Wheeler coordinate $dz = -dr/f(r)$, defined on $(0,\infty)_z$ --- see Section \ref{sect:reggewheeler}. Note that in this coordinate the $(t,r)$ part of the metric becomes conformally flat.

Here we are measuring quantities in units of the curvature radius $l$, related to the cosmological constant by $l^2 = -\frac{d(d-1)}{2\Lambda}$. Setting $\hat{r} = l r$, $\hat{t} = l t$ and then making the conformal change $\hat{g} = l^2 g$, we have
\[
\hat{g} = -\hat{f}(\hat{r}) d\hat{t}^2 + \hat{f}(\hat{r})^{-1} d\hat{r}^2 + \hat{r}^2 d\Omega_{d-1}^2
\]
where $\hat{f}(\hat{r})$ is given by
\[
\hat{f}(\hat{r}) = \frac{\hat{r}^2}{l^2} + 1 - \frac{\hat{\mu}}{\hat{r}^{d-2}}
\]
for an appropriate $\hat{\mu}$ -- this is the usual expression for the Schwarzschild--AdS metric. In this representation, the constant $\hat{\mu}$ is related to the mass $M$ of the black hole by 
\[
M = \frac{(d-1)A_{d-1}}{16\pi}\hat{\mu},
\]
where $A_{d-1}$ is the volume of the unit $d-1$ sphere. 
  
Consider a scalar field $\Psi$ with mass-squared $m^2$ propagating in a Schwarzschild--AdS background.  We allow $m^2$ to be negative but assume that it lies above the Breitenlohner-Freedman bound, namely
\[
m^2 > m^2_\mathrm{BF} = -\dfrac{d^2}{4}.
\]
The mass threshhold $m^2_\mathrm{BF}$ is related to the stability of the scalar field under small fluctuations \cite{Breitenlohner:1982} and ensures the existence of a positive energy for the Klein--Gordon equation \cite{Holzegel:2009}. In that case if we define $\nu^2 = m^2 + \frac{d^2}{4}$ then $\nu > 0$. Some of our results also apply when $\nu=0$ but we exclude this case for simplicity.

Assuming a harmonic time dependence, the Klein--Gordon equation written in the Regge-Wheeler coordinate reduces to a scattering problem on $(0,\infty)_z$ by an exponentially decaying potential (Sections \ref{sect:reggewheeler} and \ref{sect:potential}).  By a (exponentially accurate) quasimode for this problem we mean a sequence of pairs 
\[
(u_\ell, \omega^\sharp_\ell)\in C_c^\infty([0,\infty))\times \mathbb{R}, \quad \ell\geq \ell_0,
\] 
where $u_\ell(z)$ solves the scattering problem (Equation \eqref{eq:1dschrodinger}) at energy $(\omega^\sharp_\ell)^2$ up to an error of size $O(e^{-\ell/C})$ (Section \ref{sect:quasimodes}).       

\medskip
\noindent
{\bf Main Theorem.} {\em Fix $A$ satisfying
\[
1< A < \left(1+\left(\frac{2}{\mu d}\right)^{\frac{2}{d-2}}\left(\frac{d-2}{d}\right)\right)^{1/2},
\]
and let $p = \ell - 1 + d/2$. There is an $\ell_0$ such that for each angular momentum $\ell \geq \ell_0$ there exist $m(\ell)$-many quasimodes 
\[
(u_{n,\ell},\omega^\sharp_{n,\ell}),\quad n= 1,\ldots, m(\ell), \quad 1\leq m(\ell) = O(\ell),
\]
satisfying $\omega^\sharp_{n,\ell}\in p[1,A]$. Moreover, for each fixed angular momentum $\ell \geq \ell_0$, there is a one-to-one correspondence between the $\omega^\sharp_{n,\ell}$ and quasinormal modes $\omega_{n,\ell}$ in the corresponding space of spherical harmonics, satisfying
\[
\omega_{n,\ell} = \omega^\sharp_{n,\ell} + \epsilon_{n,\ell}, \quad |\epsilon_{n,\ell}| \leq e^{-\ell/C_2}, \quad n= 1,\ldots, m(\ell).
\]
The constants $\ell_0,C_1,C_2$ all depend on $A$.

In addition, if $n\geq 0$ is fixed then we have an asymptotic expansion for the real part of the quasinormal mode in powers of $\ell^{-1/2}$,
\[
\Re \omega_{n,\ell} \sim \ell + (2n +\nu +d/2) + c_{n,1} \ell^{-1/2} + c_{n,2} \ell^{-1} + \ldots, \quad \ell \geq \ell_1 = \ell_1(n).
\]
When $d=3$ we have $c_{n,1}\neq0$; when $d=4$ we have $c_{n,1}=0$ and $c_{n,2}\neq 0$; when $d\geq 5$ we have $c_{n,1}=0$ and $c_{n,2}=0$. 
}

\medskip

The basic idea behind the construction of quasinormal modes is the existence of a potential well near spatial infinity separated from the black hole horizon by a barrier -- see Figures \ref{f:1} and \ref{f:2}. We consider a related problem supporting bound states by imposing an additional Dirichlet boundary condition in the barrier; by systematically employing the exponential decay of these states in the barrier, we construct quasimodes for the original problem. Finally, the asymptotic expansion (Section \ref{sect:expansion}) is established by identifying the Schr\"odinger operator as a harmonic oscillator plus a perturbation and constructing a harmonic approximation. Although the perturbation is not globally small, we again make use of the exponential decay of various eigenfunctions; the coefficients in the expansion are ordinary Rayleigh--Schr\"odinger coefficients. The statement about the vanishing of certain coefficients verifies (in small dimensions) a recent conjecture of Dias et al. \cite{Dias:2012}.

Existence of quasimodes has been proved independently by Holzegel and Smulevici \cite{Holzegel:2013}, more generally for Kerr--AdS black holes. By Duhamel's formula, Holzegel--Smulevici use these quasimodes to show a logarithmic lower bound for the decay rate in time of solutions to the Klein--Gordon equation. This also follows from our 
construction in the case of Schwarzschild--AdS black holes, and in a forthcoming paper \cite{Gannot2} we will show how the methods of Nakamura--Stefanov--Zworski \cite{Nakamura:2002}
give expansions of solutions to the Klein--Gordon equation in terms of resonances. We also remark that logarithmic upper bounds have already been established by Holzegel--Smulevici \cite{Holzegel:2011}, and hence this represents an optimal result.

Since our quasimode construction amounts to solving an ODE of Sturm--Liouville type, we can apply a robust numerical solver \cite{Bailey:2001} to compute the associated quasimodes to high precision, even for large values of $\ell \sim 10^4$ . We find excellent agreement between these numerically computed values and the ones computed via the asymptotic expansion, with the error behaving as predicted by Proposition~\ref{thm:expansion}. 

\section{Black holes in Anti-de Sitter spacetime}

\subsection{Klein--Gordon equation}
The scalar field $\Psi$ is a solution to the Klein\textendash{}Gordon equation
\begin{equation} \label{eq:kleingordon}
(\Box_g - m^2)\Psi = 0.
\end{equation}
To compute $\Box_g$, choose coordinates $(\sigma_1,\ldots,\sigma_{d-1})$ on $S^{d-1}$ and verify that
\begin{align*}
& \frac{1}{\sqrt{-g}}\partial_{\sigma_{i}}(g^{\sigma_i \sigma_j} \sqrt{-g}\, \partial_{\sigma_{j}}) = \frac{1}{r^2}\Delta_{S^{d-1}},  \\[.5ex]
& \frac{1}{\sqrt{-g}}\partial_{t}(g^{tt} \sqrt{-g}\, \partial_{t}) = -\frac{1}{f}\partial_t^2, \\[.5ex]
& \frac{1}{\sqrt{-g}}\partial_{r}(g^{rr} \sqrt{-g}\, \partial_{r}) = \frac{1}{r^{d-1}}\partial_r(r^{d-1}f\partial_r). 
\end{align*}
Therefore 
\begin{equation}
\Box_g = -\frac{1}{f}\partial_t^2 + \frac{1}{r^{d-1}}\partial_r(r^{d-1}f\partial_r) + \frac{1}{r^2}\Delta_{S^{d-1}}.
\end{equation}

In order to solve \eqref{eq:kleingordon} we expand $\Psi$ in spherical harmonics.  Let $Y_{\ell,j}$ be a spherical harmonic with eigenvalue $-\ell(\ell+d-2)$ and consider the ansatz
\[
\Psi(t,r,\sigma;\ell,j,\omega) = r^{\frac{-d+1}{2}} e^{-i\omega t} \ Y_{\ell,j}(\sigma) \ \psi(r;\ell,\omega).
\]
 Applying $\left(\Box_g - m^2\right)$ to $\Psi$, we see that $\psi$ must satisfy the equation 
\begin{equation} \label{eq:sturmliouville}
f\DR\left(f\DR\psi\right)- f\left(\frac{(2\ell+d-2)^2-1}{4r^2} + \nu^2-\frac14 + \frac{\mu(d-1)^2}{4r^d} \right)\psi = -\omega^2\psi.
\end{equation}
for $r\in(r_0,\infty)$.
Dividing both sides by $f$ brings the equation into familiar Sturm--Liouville form. 

\subsection{Reduction to the Schr\"odinger equation} \label{sect:reggewheeler}

Define the \emph{Regge--Wheeler coordinate} by the formula 
\begin{equation} 
z(r)=\int_{r}^\infty \frac{dt}{f(t)}.
\end{equation}
This choice ensures that
\begin{equation} \label{eq:RW}
f\DR\left(f\DR\right) = \DDZ,
\end{equation}
which reduces \eqref{eq:sturmliouville} to a Schr\"odinger equation. First we record some basic observations: $r\mapsto z(r)$ maps $(r_+, \infty)$ analytically onto $(0,\infty)$ with $z(r_+)=\infty$ and and $z(\infty)=0$. In particular we have:

\begin{lem} \label{lem:rasymptotics}
The inverse $z\mapsto r(z)$ satisfies $r(z) = \dfrac{1}{z} - \dfrac{z}{3} + O(z^2) \textrm{ as } z\rightarrow 0$ and $r(z) = r_+ + O(e^{-\gamma z}) \textrm{ as } z\rightarrow \infty$ for some $\gamma>0$. Both of these asymptotics are differentiable.
\end{lem}
\begin{proof}
Since 
\[
\dfrac{1}{f(r)} = \dfrac{1}{r^2+1-\frac{\mu}{r^{d-2}}} = \dfrac{1}{r^2} -\frac{1}{r^4} + O\left(\dfrac{1}{r^5}\right)
\]
near $r=\infty$, we have
$z(r) = \dfrac{1}{r} - \dfrac{1}{3r^3} +  O\left(\dfrac{1}{r^4}\right)$ also near $r=\infty$ and hence $r(z) = \dfrac{1}{z} - \dfrac{z}{3} + O(z^2) \textrm{ as } z\rightarrow 0$. On the other hand, since $r_+$ is a simple root of $f$, expand $f$ at $r_+$ and integrate to obtain 
\[
-f'(r_+)z(r) = \log(r-r_+) + G(r),
\] 
where $G(r)$ is analytic near $r=r_+$. By an application of the implicit function theorem, it follows that $r(z) = F(e^{-f'(r_+)z})$ with $F$ analytic near zero and $F(0) = r_+$. The result follows with $\gamma = f'(r_+) > 0$.    
\end{proof}
\begin{rem} Spatial infinity corresponds to $r=\infty$ while the event horizon corresponds to $z=\infty$. Since from now on we will mostly use the Regge-Wheeler coordinate, we stress that ``infinity'' will refer to $z=\infty$ unless stated otherwise.
\end{rem}

Using \eqref{eq:RW}, we see the function $z\mapsto \psi(r(z))$ must satisfy the one-dimensional Schr\"odinger equation
\begin{equation} \label{eq:1dschrodinger}
\left(-\DDZ + V_{\mathrm{eff}}(z;\ell) -\omega^2\right)\psi(r(z)) = 0
\end{equation}
for $z\in(0,\infty)$, with the effective potential
\begin{equation} \label{eq:eff}
V_{\mathrm{eff}}(z;\ell)= f(r(z))\left(\frac{(2\ell+d-2)^2-1}{4r(z)^2} + \nu^2-\frac{1}{4}+ \frac{\mu(d-1)^2}{4r(z)^d}\right),
\end{equation}
and spectral parameter $\omega^2$.

\subsection{Analysis of the effective potential} \label{sect:potential}
To study the large angular momentum limit, introduce a semiclassical parameter 
\begin{equation}
h^{-1}=\frac{(2\ell+d-2)}{2}
\end{equation}
so that $h\rightarrow 0$ as $\ell\rightarrow\infty$. Multiplying Equation \eqref{eq:1dschrodinger} by $h^2$ results in a semiclassical Schr\"odinger equation. Define a new potential and spectral parameter by
\[
V(z;h) = h^2 V_\mathrm{eff}(z; \ell), \quad E(h) = h^2\omega^2.
\]
Then $z\mapsto \psi(r(z))$ satisfies Equation \eqref{eq:1dschrodinger} if and only if it satisfies
\begin{equation} \label{eq:rescaledschrodinger}
\left(-h^2 \DDZ+V(z;h)-E(h)\right)\psi(r(z)) = 0.
\end{equation}
We will continue to refer to $V(z;h)$ as the effective potential. 

\begin{lem} \label{prop:Vinfinity}
The effective potential $V$ satisfies $\frac{d^k}{dz^k} V(z;h) = O\left(e^{-\gamma z}\right),\, k\geq 0$, uniformly in $h$ as $z\rightarrow \infty$.\end{lem}
\begin{proof} Using Lemma \ref{lem:rasymptotics} we see that $f(r(z)) = f(r_+ + O(e^{-\gamma z})) = O(e^{-\gamma z})$ for large $z$. Since the asymptotics of $r(z)$ can be differentiated, it also follows that $\frac{d^k}{dz^k} f(r(z)) =O\left(e^{-\gamma z}\right)$ for large $z$. But $V$ is the product of $f$ and 
\[
r^{-2} -4 h^2 r^{-2} + h^2(\nu^2-1/4)+ 4 h^2 \mu(d-1)^2 r^{-d}.
\]
which is uniformly bounded in $z$ and $h$ along with all of its derivatives. It remains to apply the Leibniz rule.
\end{proof}

Different decompositions of the effective potential are useful; with respect to the $r$-coordinate, the most natural is
\[
V(z;h) = V_{-1}(z;h) + V_0(z) + h^2 V_1(z),
\]
where
\[
V_{-1} = h^2\left(\nu^2-\frac{1}{4}\right)f,\quad V_0 = 1 + \frac{1}{r^2} -\frac{\mu}{r^d}, \quad V_1 = \left(-\frac{1}{4r^2} +\frac{\mu(d-1)^2}{4r^d}\right)f.
\]
Here $V_{-1}$ is the analogue of the usual centrifugal term which appears in spherically symmetric problems after separation of variables, in the sense that it behaves as $h^2 (\nu^2-1/4)z^{-2}$ as $z\rightarrow 0$. In the same parlance, $V_0$ plays the role the physical potential, while $V_1$ is uniformly bounded and hence $h^2 V_1$ is globally a lower order term in $h$. On the other hand, from the scattering point of view it is natural to consider $-h^2\DDZ + h^2(\nu^2-1/4)z^{-2}$ as the unperturbed, or free, operator. We therefore define the perturbation $W(z;h)$ by the formula
\[
V(z;h) =  h^2(\nu^2-1/4)z^{-2} + W(z;h).
\]
In light of Lemmas \ref{lem:rasymptotics} and \ref{prop:Vinfinity}, we see that $W$ is smooth, uniformly bounded in $h$, and decays as $z\rightarrow \infty$ like an inverse square when $\nu \neq 1/2$ or exponentially when $\nu =1/2$. Although $W$ is not in general exponentially decaying, it will be useful in some auxiliary results. 

For use in the perturbation expansion of low-lying quasimodes, we also record the following:
\begin{lem} \label{prop:Vasymptotics}
The effective potential can be written as
\[
V(z;h) = h^2\left(\frac{\nu^2-\frac14}{z^2}\right)+1+z^2+ R(z;h),
\]
where $R(z;h) = O(z^3) + h^2 O(1)$ for $z$ in a compact set. 
\end{lem}
\begin{proof} Using Lemma \ref{lem:rasymptotics} we see that near $z=0$,
\[
V(z;h) = h^2\left(\frac{\nu^2-\frac14}{z^2}\right)+1 + z^2 + h^2 \left(\frac{\nu^2-1}{3}\right) + h^2 O(z) + O(z^3).
\]
\end{proof}

The behavior of the effective potential near the origin depends on the value of $\nu$. For example $V$ is repulsive if $\nu > 1/2$ and weakly attractive when $0 < \nu \leq 1/2$. The case $\nu = 1/2$ is the conformally coupled case. Despite the different pointwise behavior of the centrifugal term $V_{-1}$, by a Hardy inequality we are able to treat all values of $\nu > 0$ on equal footing. Therefore, we will mostly be concerned with the structure of the physical potential $V_0$. We have $V_0(z) > 0$ and clearly $V_0(0) = 1$, $V_0(z)\rightarrow 0$ as $z\rightarrow \infty$.  

\begin{figure}[htb]
\centering
\includegraphics[width=\textwidth]{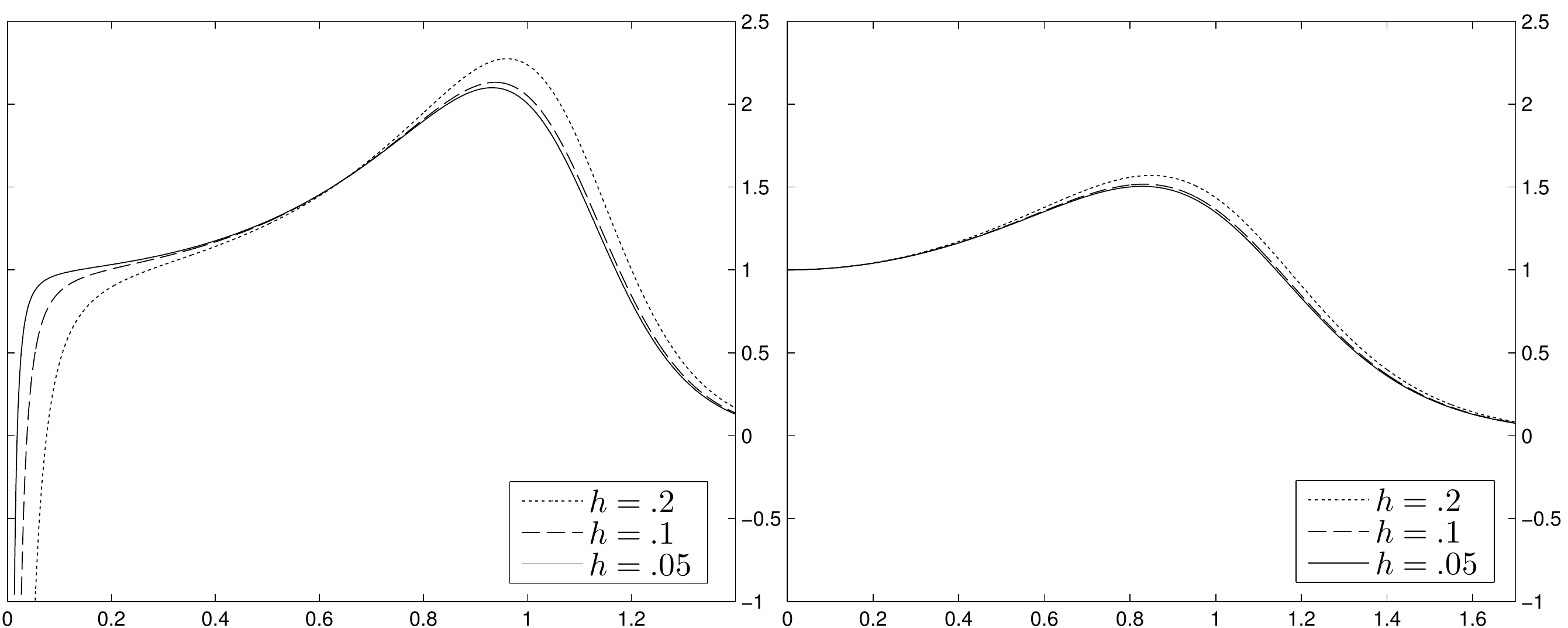}
\caption{Plots of $z$ versus $V(z;h)$ for different values of $d,\mu,\nu,h$. Left: $d=6,\, \mu =1/8,\, \nu = \sqrt{3/28}$. Right: $d=4,\,\mu=1/2,\,\nu = 1/2$. See Figure \ref{f:2} for a plot when $\nu > 1/2$.} \label{f:1}
\end{figure}

\begin{lem} The physical potential $V_0$ has a unique nondegerate local maximum satisfying 
	\[
	z_\mathrm{max} = z\left(\left(\frac{\mu d}{2}\right)^{\frac{1}{d-2}}\right), \qquad V_0(z_\mathrm{max}) = 1+\left(\frac{2}{\mu d}\right)^{\frac{2}{d-2}}\left(\frac{d-2}{d}\right),
	\]
	and no other local extrema for $z\in (0,\infty)$.
\end{lem}
\begin{proof}
To find the extrema of $V_0(z)$, it suffices to find the roots of
\[
\DR V_0(z(r)) = - \frac{2}{r^3} + \frac{\mu d}{r^{d+1}}
\]
for $r\in (r_+,\infty)$.
\end{proof}

The existence of this local maximum is related to the trapping of null-geodesics on the background \cite{Holzegel:2011}. Next, we examine turning points. By the previous lemma, for any real $1<E<V_0(z_\mathrm{max})$ the equation $V_0(z)-E=0$ has two solutions, see Figure \ref{f:2}. We will denote these two turning points as $z_A(E)$ and $z_B(E)$ where $z_A(E)<z_B(E)$. Clearly when $E$ is independent of $h$, so are $z_A(E)$ and $z_B(E)$ --- they are given by $z(r_A(E))$ and $z(r_B(E))$ where $r_A(E) > r_B(E)$ are the real solutions to
\[
1 - E + \frac{1}{r^2} -\frac{\mu}{r^d} = 0.
\]
We are also interested in those energies $E$ satisfying $E = 1+Th$ for fixed $T>0$ and $h$ small enough, since at these energy levels the harmonic approximation (Proposition \ref{thm:h1/2}) is valid.
\begin{lem} \label{lem:turningpoints}
Suppose $E = 1+Th$ where $T>0$ is independent of $h$. There exists $h_0>0$ and positive constants $z'_{A}(T),z'_{B}$ such that if $h\in(0,h_0)$ then the following is true: 
	\[
	 z_A(1+Th) = z'_{A}(T)h^{1/2} + O(h), \quad z_B(1+Th) = z'_{B} + O(h).
	\]
	Furthermore, $z'_{A}(T) = T^{1/2}$ and $z'_{B} = z\left(\mu^{\frac{1}{d-2}}\right)$.
\end{lem}
\begin{proof}
Set $g(r,h) = r^{-2} - \mu r^{-d} - Th$. Then $r=\mu^{\frac{1}{d-2}}$ is a simple root of $g(r,0)$, so we may apply the implicit function theorem. The root of $g(r,0)$ at $r=0$ is a multiple root, so instead rescale by $\tilde{h} = h^{1/2}, \tilde{r} = \tilde{h}r$ and set $\tilde{g}(\tilde{r},\tilde{h}) = \tilde{r}^{-2} - \tilde{h}^{d-2}\mu \tilde{r}^{-d} -T$. Then $\tilde{g}(\tilde{r},0)$ has a simple root at $\tilde{r} = T^{-1/2}$. The proof is finished by an application of the implicit function theorem and the asymptotics of $z(r)$ for large $r$.
\end{proof} 

\begin{figure}[htpb]
\hbox{\hspace{-25ex}\includegraphics[scale=0.8]{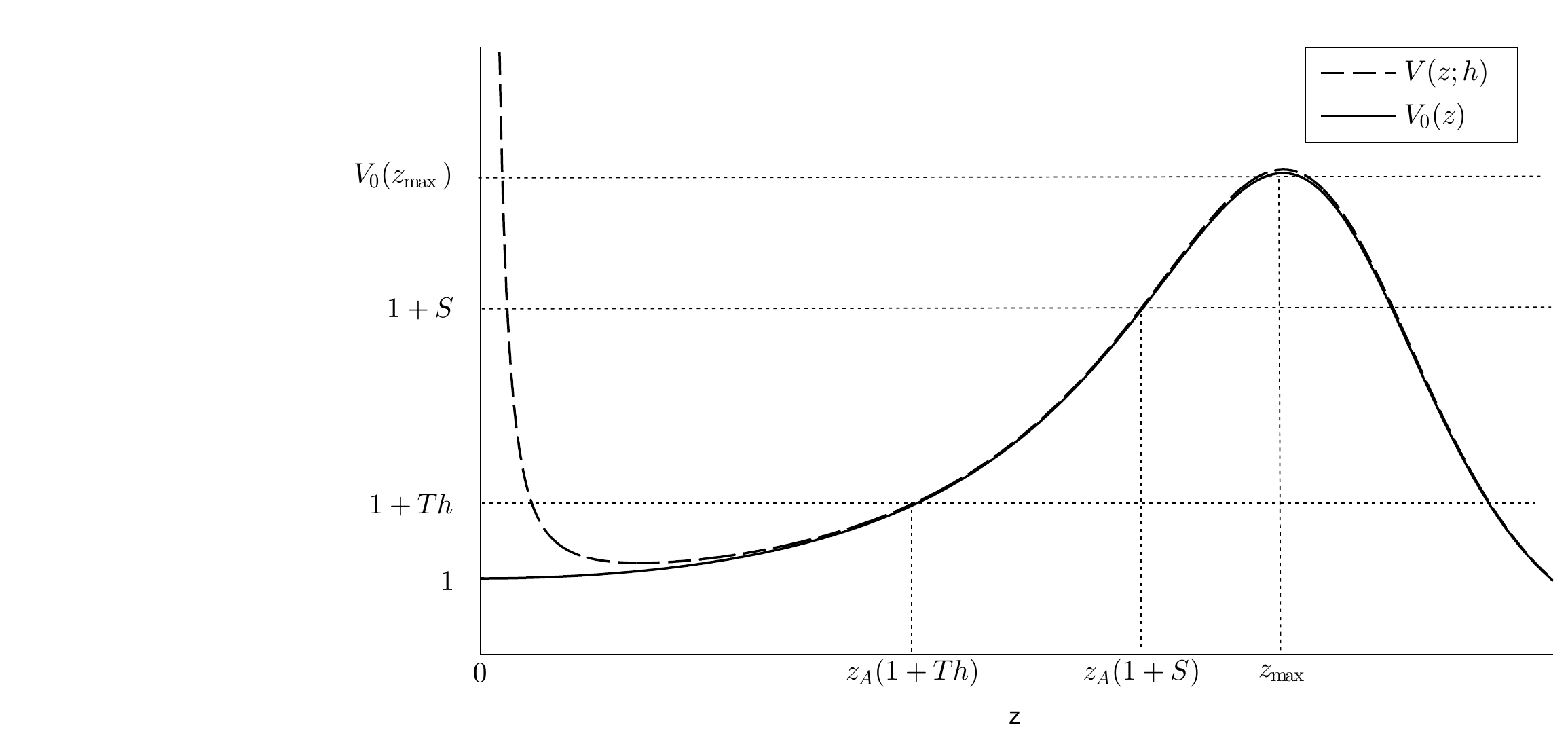}}
\caption{A schematic plot of $V_0$ and $V$ for $\nu > 1/2$ illustrating the maximum and the left-most turning points.}
\label{f:2}
\end{figure}

\section{Quasimodes}

\subsection{Self-adjoint realizations}
Our first goal is to give a Hilbert space formulation of the resonance problem. In other words, we are interested in choosing a suitable self-adjoint realization $P(h)$ of $-h^2\DDZ + V$ on $L^2(0,\infty)$. Then, in order to construct quasimodes for $P(h)$ we also realize $-h^2\DDZ + V$ as a self-adjoint \emph{reference operator} $P^\sharp(h)$ on $L^2(0,z_\mathrm{max})$ with discrete spectrum. Each eigenfunction of $P^\sharp(h)$ will give rise to a quasimode for $P(h)$. We therefore begin by discussing self-adjoint realizations of $-h^2\DDZ + V$ on an arbitrary interval $J = (0,c)$, where $0<c\leq \infty$. Of course the whole subtlety here lies in that $V$ has a singularity at the origin ---  for all the material in this section, we refer to the books \cite{Zettl}, \cite{Weidmann} where exhaustive treatments of singular Sturm--Liouville operators can be found. Since $W$ is analytic at the origin, the classical Frobenius theory for ordinary differential equations applies. The regular singular point at the origin has indicial roots $\nu_+ = 1/2+\nu$ and $\nu_- = 1/2-\nu$ and hence for $\nu > 0$ the equation $-h^2 u'' + Vu=0$ has linearly independent solutions of the form
\[
u_+ = z^{\nu_+}\tilde{u}_+, \quad u_- = -\frac{1}{2\nu}z^{\nu_-}\tilde{u}_-,
\]
where $\tilde{u}_+, \tilde{u}_-$ are analytic and $\tilde{u}_+(0) = \tilde{u}_-(0) = 1$. The normalizations are chosen so that their Wronskian is one. When $\nu \geq 1$, only $u_+$ is square-integrable near the origin, while both $u_+$ and $u_-$ are square-integrable if $0 < \nu < 1$. This dichotomy corresponds to the fact that a boundary condition at $z=0$ must be imposed when $0 <\nu<1$, but not when $\nu \geq 1$. Different boundary conditions have been considered in the physics literature --- for a classical discussion see \cite{Avis:1978}. In this note we only handle the case of a Dirichlet-like condition, but see \cite{Warnick:2013}, \cite{Bachelot:2008} for two recent works considering a wider range of boundary conditions. 

More precisely, define the minimal operator $P_\mathrm{min}(h)$ with domain $D_\mathrm{min}$ as the closure of the expression $-h^2\DDZ + V$ on $C_c^\infty(J)$. The corresponding maximal operator $P_\mathrm{max}$ is given by the same expression on the domain 
\[ 
D_\mathrm{max}(J) = \left\{u\in L^2(J): u,\,u'\in AC(J),\,-h^2 u'' + Vu \in L^2(J)\right\}.
\]
Since $-h^2 u'' + Vu \in L^2(J)$ is equivalent to $-u'' + (\nu^2-1/4)z^{-2}u \in L^2(J)$ by the boundedness of $W$, this set is independent of $h$. It is well known that $P_\mathrm{min}(h)^* =P_\mathrm{max}(h)$ and $P_\mathrm{max}(h)^* =P_\mathrm{min}(h)$ . The following observations on the structure of the maximal domain are classical:

\begin{lem}
Suppose $u\in D_\mathrm{max}(J)$. Then there exist constants $b_+(u), b_-(u)$ and an absolutely continuous function $\tilde{u}$ with the property that $u=b_+(u)u_+ + b_-(u)u_- + \tilde{u}$. Furthermore, $\tilde{u}$ satisfies
\begin{enumerate}
	\item $\lim_{z\rightarrow 0^+} z^{-1/2}\tilde{u}(z) = 0$ and $\lim_{z\rightarrow 0^+} \tilde{u}'(z) = 0$.
	\item $z^{-1}\tilde{u}$ is square integrable near $z=0$.
	\item $\tilde{u}'$ is square integrable near $z=0$.
\end{enumerate}
\end{lem}

\begin{proof}
Since $v=-h^2 u'' + Vu \in L^2(J)$, by variation of parameters we have
\[
u = b_+(u) u_+ + b_-(u) u_- + \tilde{u}
\]
where
\[
\tilde{u} = u_+(z) \int_a^z u_-(t) v(t) dt - u_-(z) \int_0^z u_+(t) v(t) dt.
\]
When $\nu\geq 1$, set $a=1$, and when $0<\nu<1$ set $a=0$. It then follows by Cauchy--Schwarz that
\begin{align*}
\tilde{u}(z) = O\left(z^{3/2}\right), &\quad \tilde{u}'(z) = O\left(z^{1/2}\right),\; \nu > 0, \nu \neq 1,\\
\tilde{u}(z) = O\left(z^{3/2}\log(z)^{1/2}\right),&\quad \tilde{u}'(z) = O\left(z^{1/2} \log(z)^{1/2}\right),\; \nu = 1.
\end{align*}
The properties of $\tilde{u}$ immediately follow. 
\end{proof}

The linear functionals $b_+, b_-$ are referred to as boundary conditions. Since $u\in L^2(J)$, we see that $b_-(u) = 0$ if $\nu  \geq 1$. On the other hand, when $0<\nu<1$, the most general (separated) boundary condition at the origin is of the form
\[
\sin (\theta) b_+(u) + \cos (\theta) b_-(u) = 0, \theta \in [0,\pi). 
\]
In this paper we take the Dirichlet-like boundary condition $b_-(u) = 0$. This is easily seen to be equivalent to
\[
\lim_{z\rightarrow 0}z^{\nu -1/2}u =0.
\]
\begin{rem}
When $\nu=1/2$ the singularity vanishes and we have an ordinary Dirichlet condition; in fact, for all $\nu >0$ this boundary condition corresponds to the Friedrichs extension of $P_\mathrm{min}(h)$ \cite{Zettl} (we will comment on the semiboundedness shortly). 
\end{rem}

Summarizing the above discussion, we have established the following.
\begin{cor} \label{cor:domainstructure}
Suppose $u\in \left\{u\in D_\mathrm{max}(J): \lim_{z\rightarrow 0^{+}} z^{\nu -1/2}u = 0\right\}$. Then
\begin{enumerate}
	\item $\lim_{z\rightarrow 0^+} z^{-1/2}u(z) = 0$.
	\item $\lim_{z\rightarrow 0^+} u(z)u'(z) = 0$.
	\item $z^{-1}u$ is square integrable near $z=0$.
	\item $u'$ is square integrable near $z=0$.
\end{enumerate}
\end{cor}
\begin{proof}
This follows immediately from the previous lemma combined with $b_-(u) = 0$.
\end{proof}

Next we discuss the semiboundedness of $-h^2 \DDZ + V$ on the interval $J$. The natural approach here is to use a weighted Hardy inequality; the use of such inequalities in the study of massive wave equations on Kerr--AdS backgrounds was pioneered in \cite{Holzegel:2009}, \cite{Holzegel:2011}. We can use a version of the classical ``factorization method'' \cite{Hull:1951} to prove such results: given a second order self-adjoint operator $A$, find a (non self-adjoint) first order operator $B$ and a number $\beta$ with the property that $A \geq B^*B + \beta$. See also \cite{Holzegel:2012} for a similar approach in the current context.

\begin{lem} \label{lem:factorization}
Suppose $u\in\left\{ u\in D_\mathrm{max}(J): \lim_{z\rightarrow 0^+} z^{\nu-1/2} u(z) = 0,\, u(c)=0 \right\}$. Let $Y$ be a smooth bounded function with bounded derivative on $J$, satisfying $Y(z) = O(z), Y'(z) = O(1)$ as $z\rightarrow 0^+$.  Then
\[
\lVert h D_z u - i f Y u \rVert^2_{L^2(J,dz)} = \left < h^2 D_z^2 u, u \right >_{L^2(J,dz)} + \left < (fY)^2 u - h f \partial_r(fY)u,u \right >_{L^2(J,dz)}.    
\]
\end{lem}
\begin{proof}
Here we are writing $D_z = -i \partial_z$. Integrate by parts and recall that $\partial_r = -f^{-1} \partial_z$. The integration by parts is justified by using Corollary \ref{cor:domainstructure} near $z=0$ and the vanishing of $u$ at $z=c$.
\end{proof} 
The following first appeared in \cite{Holzegel:2011}; we offer an alternative proof.
\begin{lem} [{\cite[Lemma 7.1]{Holzegel:2011}}] \label{lem:weightedhardy} 
Suppose 
\[
u\in\left\{ u\in D_\mathrm{max}(J): \lim_{z\rightarrow 0^+} z^{\nu-1/2} u(z) = 0,\, u(c)=0 \right\}.
\] 
Then
\[
\left < -h^2 \textstyle{\DDZ} u + V_{-1}u , u \right >_{L^2(J,dz)} \geq 0.
\]
\end{lem}
\begin{proof}
Recall that $V_{-1} = h^2 (\nu^2 -1/4) f \geq -h^2 f/4$. We therefore want to find $Y$ so that $(fY)^2 - h f \partial_r(fY) \leq -h^2 f/4$. Set $Y = h(r-r_+)f^{-1}/2$. Then $Y(z) = O(z), Y'(z) = O(1)$ as $z\rightarrow 0^+$. Furthermore, 
\[
f^2 Y^2 - h f\partial_r (fY) = \frac{h^2(r^2-2rr_+ + r_+^2)}{4} - \frac{h^2f}{2}.
\]
But it is easy to see that this quantity does not exceed $-\textstyle{\frac{1}{4}}h^2f$ for $z\in(0,\infty)$ or equivalently $r\in (r_+,\infty)$. Indeed, that is equivalent to
\[
r^2 - 2rr_+ + r_+^2 \leq r^2 + 1 - \mu r^{2-d},
\] 
or $-2rr_+ + r_+^2 \leq 1 - \mu r^{2-d}$. But both sides assume the same value of $-r_+^2$ when $r=r_+$, while at the same time the left hand side is decreasing and the right hand side is increasing as $r$ increases. The result follows by an application of Lemma \ref{lem:factorization}.
\end{proof}
Note that this result does not rely on the smallness of $h$. Furthermore, for each $\ell \geq 0$ we clearly have $V_0 + h^2 V_1 > 0$. We thus define $P(h)$ as the operator $-h^2 \DDZ + V$ with domain
\[
\DOM = \left\{u\in D_\mathrm{max}(0,\infty): \lim_{z\rightarrow 0^+} z^{\nu-1/2} u(z) = 0 \right\}.
\]
Then $P(h)$ is self-adjoint and $P(h) \geq 0$. Also define the Bessel operator $L_\nu(h)$ as $-h^2\DDZ + h^2(\nu^2-1/4)z^{-2}$ acting on $\DOM$. It is well known that $L_\nu(h) \geq 0$ (the usual Hardy inequality) and that $\sigma(L_\nu(h)) = \sigma_\mathrm{ess}(L_\nu(h)) = [0,\infty)$ \cite{Everitt:2007}.

\begin{prop} The spectrum of $P(h)$ is purely absolutely continuous and equal to $[0,\infty)$.
\end{prop}
\begin{proof}
We have $\sigma_\mathrm{ess}(P(h)) = [0,\infty)$ since $P(h)$ is a relatively compact perturbation of $L_\nu(h)$, see the proof of Proposition \ref{prop:relativecompactness}. But we also know that $\sigma(P(h)) \subseteq [0,\infty)$ from $P(h)\geq 0$. For a nice proof of the absolute continuity of the nonnegative spectrum using one-dimensional techniques, see \cite[Theorem 15.3]{Weidmann}.
\end{proof}

Next we turn to the construction of the reference operator. Set
\[
\Omega = (0,z_\mathrm{max}].
\]
Define $P^\sharp(h)$ to be the self-adjoint operator $-h^2 \DDZ + V$ with domain
\[
\DOM^\sharp = \left\{u\in D_\mathrm{max}(\Omega): \lim_{z\rightarrow 0^+} z^{\nu-1/2} u(z) = 0, \, u(z_\mathrm{max}) = 0  \right\}.
\]
Correspondingly, define $L^\sharp_\nu(h)$ as $-h^2\DDZ + h^2(\nu^2-1/4)z^{-2}$ acting on $\DOM^\sharp$. It is well known that $L^\sharp_\nu(h)$ has purely discrete spectrum with eigenvectors given by spherical Bessel functions, see Proposition \ref{prop:weyllaw}. The following spectral properties of $P^\sharp(h)$ do not follow from any general theory, again owing to the singular endpoint at $z=0$ --- see \cite[p. 208]{Zettl} for a summary of the possible spectral behavior. 
\begin{prop} The spectrum of $P^\sharp(h)$ is purely discrete. The eigenvalues are all simple and can be arranged as
\[
0 \leq E^\sharp_0(h) < E^\sharp_1(h) < E^\sharp_2(h) < \ldots
\] 
\end{prop}
\begin{proof}
Since $W$ is bounded and $L^\sharp_\nu(h)$ has compact resolvent, it follows that $P^\sharp(h)$ is again a relatively compact perturbation of $L^\sharp_\nu(h)$, and hence $P^\sharp(h)$ has no essential spectrum. Finally, since $z_\mathrm{max}$ is a regular endpoint, the eigenvalues of $P^\sharp(h)$ are all simple by the usual argument. 
\end{proof}
The corresponding eigenvectors will be denoted $u_n^\sharp(h)$. Using Lemma \ref{lem:factorization} we can show that the spectrum of $P^\sharp(h)$ is separated from the minimum of the potential.

\begin{lem}
There exists $C>0$ and $h_0 > 0$ such that $P^\sharp(h) \geq 1 + C h$ for all $h\in (0,h_0)$.
\end{lem}
\begin{proof}
Writing $Y = Y_0 f^{-1} + h Y_1 f^{-1}$ and collecting powers of $h$, it suffices to find $Y_0$ and $Y_1$ satisfying
\[
Y_0^2 \leq V_0 - 1, \quad 2Y_0 Y_1 - f \partial_{r} Y_0 \leq - D , \quad Y_1^2 - f \partial_{r} Y_1 \leq -f/4
\]
on $\Omega$, for some $D>0$. We would then have
\[
(1+Dh) \|u\|^2_{L^2(\Omega)} \leq \langle (-h^2 \textstyle{\DDZ} + V_{-1} + V_0 )u,u \rangle_{L^2(\Omega)},
\] 
and the result would follow since $V_1$ is bounded. So let $\delta = \left(1-\frac{2}{d}\right)^{1/2}$, and set $Y_0 = -\delta r^{-1}$ and $Y_1 = r/2$. An easy calculation shows that $Y_0^2 \leq r^{-2} f - 1$ and $Y_1^2 - f \partial_r Y_1 \leq -f/4$ for $r \geq r_\mathrm{max} = \left(\frac{\mu d}{2}\right)^{\frac{1}{d-2}}$. Finally, compute
\[
2 Y_0 Y_1 - f \partial_r Y_0 = -2 \delta - \delta r^{-2} + \delta \mu r^{-d} \leq -2\delta
\]
for $r\geq r_\mathrm{max}$, and set $D = 2\delta$.
\end{proof}

\begin{rem} The simple choice of $Y$ above is sufficient to show that the first eigenvalue is separated from the minimum of the potential, but the value of $C$ given in the proof is not optimal. Later we will give a full asymptotic expansion for the first eigenvalue which shows that $C=2$ is the correct value; it is likely that a more refined choice of $Y$ could recover this value.
\end{rem}

Later we will need   

\begin{lem} \label{lem:W(h)}
There exists $h_0>0$ such that $P^\sharp(h) \geq L^\sharp_\nu(h)$ for all $h \in (0,h_0)$.
\end{lem}
\begin{proof}
It suffices to show that $V(z;h)> h^2(\nu^2-1/4)z^{-2}$ on $\Omega$. First suppose $0 < \nu < 1/2$. Let $z_0(h)$ denote the solution to $V(z;h)=0$. By the method of Lemma \ref{lem:turningpoints}, it is easy to see that $z_0(h)=O(h)$ and hence on $(0,z_0(h)]$ we have $V(z;h) = h^2(\nu^2-1/4)z^{-2} + 1 + O(h^2) >h^2(\nu^2-1/4)z^{-2}$. On the other hand, $V(z;h)$ satisfies $V(z;h) > 0 > h^2(\nu^2-1/4)z^{-2}$ on $(z_0(h),z_\mathrm{max}]$. 

In the case when $\nu\geq 1/2$, let $z_1(h)$ denote the point where the minimum of $V(z;h)$ on $\Omega$ is attained. Then again we have $z_1(h) = O(h^{1/2})$ if $\nu>1/2$ or $z_1(h) = 0$ if $\nu = 1/2$. We then see that $V(z;h) = h^2(\nu^2-1/4)z^{-2} + 1 + O(h)>h^2(\nu^2-1/4)z^{-2}$ on $(0,z_\mathrm{min}(h)]$, while on the complement $V'(z;h)>0$ and $\left(h^2(\nu^2-1/4)z^{-2}\right)' < 0$ so that $V(z;h) > h^2(\nu^2-1/4)z^{-2}$. Hence in all cases we have $V(z;h)> h^2(\nu^2-1/4)z^{-2}$ on $\Omega$.
\end{proof}

Now we define a \emph{model operator} $\widetilde{P}(h)$ which locally near the origin resembles the reference operator. Let $\widetilde{P}(h)$ denote the operator 
\[
\widetilde{P}(h) = -h^2\DDZ + h^2(\nu^2-1/4)z^{-2} + z^2 + 1
\] 
on $L^2(0,\infty)$ with domain
\[
\widetilde{\DOM} = \left\{u\in \widetilde{D}_\mathrm{max}(0,\infty): \lim_{z\rightarrow 0^+} z^{\nu-1/2} u(z) = 0, \right\}.
\]
The maximal domain for $\widetilde{P}(h)$ is defined here as
\[ 
\widetilde{D}_\mathrm{max}(J) = \left\{u\in L^2(J): u,\,u'\in AC(J),\,-h^2 u'' + h^2(\nu^2-1/4)z^{-2}u + z^2u \in L^2(J)\right\}.
\]
\begin{rem}
The domain $\widetilde{\DOM}$ is independent of $h$. This is because one can show that $u\in \widetilde{D}_\mathrm{max}(0,\infty)$ actually implies $z^2 u \in L^2(0,\infty)$.
\end{rem}
\begin{rem} When $\nu$ is a nonnegative integer, $\widetilde{P}(h)$ is just the radial part of the isotropic harmonic oscillator in two dimensions, corresponding to the spherical harmonic indexed by $\nu$. 
\end{rem}

By a harmonic approximation we will identify the bottom of the spectrum $\sigma(P^\sharp(h))$ by comparing it to $\sigma(\widetilde{P}(h))$. An integration by parts for $u\in \widetilde{D}$ shows that
\[
0\leq \left\langle h u' + \left(z-\textstyle{\frac{h}{2z}}\right)u,\, hu' + \left(z-\textstyle{\frac{h}{2z}}\right)u \right\rangle = \left\langle -h^2 u'' + \left(-\textstyle{\frac{1}{4z^2}} + z^2 -2h\right)u,u \right\rangle,
\]
so that $\widetilde{P}(h) \geq 1 + 2h$. In fact the spectrum is explicitly known.
\begin{prop} The spectrum of $\widetilde{P}(h)$ is purely discrete. The eigenvalues are all simple and can be arranged as
\[
1+2h \geq \TE_0(h) < \TE_1(h) < \TE_2(h) < \ldots
\] 
Moreover the eigenvalues are given by
\[
\TE_n(h) = 1+2\left(2n+1+\nu\right)h ,\\
\]
and the normalized eigenvectors are given by 
\[
\widetilde{u}_n(z;h) = h^{-1/4}\widetilde{u}_n(h^{-1/2}z;1)
\]
where
\[
\widetilde{u}_n(z;1) =  \sqrt{\frac{2\,\Gamma(n+1+\nu)}{n!\,\Gamma(1+\nu)^2}}z^{\nu+\frac{1}{2}}e^{-\frac{z^2}{2}}{_1}F_1(-n,1+\nu,z^2).
\]
\end{prop}
Also see \cite{Hall:2000} for a detailed discussion of this operator in the context of ``spiked hamonic oscillators''. Here ${_1}F_1(a,b,y)$ is the confluent hypergeometric function; since $n$ is an integer, ${_1}F_1(-n,1+\nu,y)$ is just polynomial of degree $n$, proportional to the Laguerre polynomial $L^{(\nu)}_{n}(y)$.

\subsection{Agmon estimates} \label{sect:agmon}
The strategy for producing exponentially accurate quasimodes for $P(h)$ is to truncate an eigenfunction $u^\sharp(h)$ of $P^\sharp(h)$ through multiplication by a cutoff function $\chi$ and then extend $\chi u^\sharp(h)$ by zero as an element of $\DOM$. If $u^\sharp(h)$ and its derivative are exponentially small in $L^2$ on the support of $\chi'$ then $\chi u^\sharp(h)$ will be an exponentially accurate quasimode for $P(h)$. For certain energy levels below the maximum of $V_0$ there is a classically forbidden region where we can use Agmon-type estimates to obtain exponential decay for $u^\sharp(h)$. It then remains to choose $\chi$ with derivative supported in this region.   

Suppose $\phi\in C^\infty(\Omega)$ and $f\in \DOM^\sharp$. Then $e^{-{\phi/h}} f\in \DOM^\sharp$ and for any $E$, integration by parts gives
\begin{multline} \label{eq:agmonintbyparts}
\mathrm{Re}\,\left <e^{\phi/h} \left( -h^2\textstyle{\DDZ} + V - E\right) e^{-{\phi/h}} f,f \right>_{L^2(\Omega)} \\ = \left<\left(-h^2\textstyle{\DDZ} + V-E-\left(\phi'\right)^2\right)f,f\right>_{L^2(\Omega)}.
\end{multline}

\begin{lem} 
Suppose $\phi \in C^\infty(\Omega)$, $u\in \DOM^\sharp$ and $\chi \in C_c^\infty(\Omega)$. Then 
\begin{multline} \label{eq:agmoncommutator}
\Re \left <e^{\phi/h}(P^\sharp(h)- E)\chi u,e^{\phi/h}\chi u\right > = \Re \left< e^{\phi/h} \chi (P^\sharp(h)-E)u,e^{\phi/h}\chi u\right>  \\ + h^2\left< u, e^{2\phi/h} \left((\chi')^2 + 2 h^{-1} \phi' \chi \chi'\right)u \right>. 
\end{multline}
\end{lem}
\begin{proof}
We have 
\begin{multline*}
\left <e^{\phi/h}(P^\sharp(h) - E)\chi u,e^{\phi/h}\chi u\right > = \left< e^{\phi/h} \chi (P^\sharp(h)-E)u,e^{\phi/h}\chi u\right> \\ + \left< [\textstyle{-h^2\DDZ},\chi]u, e^{2\phi/h}\chi u \right>.
\end{multline*}
Taking real parts and integrating by parts the second term on the right hand side gives
\begin{align*}
\Re \left< [\textstyle{-h^2\DDZ},\chi]u, e^{2\phi/h}\chi u \right> &= \Re h^2 \left < -\chi'' u - 2\chi' u', e^{2\phi/h}\chi u\right > \\
&= h^2 \left< u, e^{2\phi/h} \left((\chi')^2 + 2 h^{-1} \phi' \chi \chi'\right)u \right>.
\end{align*} 
\end{proof}

For $E$ real set 
\[
\Omega^-(E) = (0,z_A(E)],\qquad \Omega^+(E) = (z_A(E),z_\mathrm{max}].
\]
Then $\Omega^+(E)$ corresponds to the classically forbidden region inside of $\Omega$ in the sense that $V_0 > E$ on $\Omega^+(E)$. 

The following two results show that if $E(h)<1+Th$ for some $T>0$ then any solution to $P^\sharp(h) u = E(h) u$ has the property that $\exp(z^2/ch) u$ is controlled by $u$ in $L^2(\Omega)$ for some $c>0$. The key here is that $z_A(1+Th) = O(h^{1/2})$, and hence $\exp(z^2/c)u$ is trivially controlled by $u$ on $(0,z_A(1+Th)]$, despite the fact that we are in the classically allowed region.

\begin{lem} \label{lem:parabola}
Let $T>0,\delta >0$. Then there exists $k > 0, h_0>0$ depending on $T,\delta$ such that 
\[
V_0(z) - (1+Th) - kz^2 > \frac{3\delta}{2} h 
\]  
for $z\in \Omega^+(1+(T+2\delta)h)$ and $h\in (0,h_0)$.
\end{lem}
\begin{proof}
Recall that $z_{A}(1+Th) = z'_{A}(T)h^{1/2} + O(h)$ with $z'_{A}(T) = T^{1/2}$. Set
\[
k = \frac{\delta}{4 z'_{A}(T+2\delta)^2} = \frac{\delta}{4(T+2\delta)}.
\]
and then define $M(z;h) = V_0(z) - (1+Th) - kz^2$. Recalling that $V_0(z) = 1 + z^2 + O(z^3)$, we can see that 
\[
M(z_{A}(1+(T+2\delta)h);h) = 2\delta h  - \frac{\delta}{4}h + O(h^{3/2}) > \frac{3\delta}{2} h  
\]
for $h$ small and $\epsilon$ small but fixed. We also see that $k < 1/8$. Thus we have, for example,
\[
M'(z_A(1+(T+2\delta)h);h) > \frac{3}{2}T^{1/2} h^{1/2} > 0.
\]
But $M''(z;h) \geq 0$ on an interval $[0,A]$ with $A > 0$ independent of $h$. Thus we can conclude that $M(z;h) > \frac{3\delta}{2}h$ on $[z_A(1+(T+2\delta)h),A]$ since $M$ is increasing there. Conversely, on $(A, z_\mathrm{max}]$ a much stronger inequality holds, namely $M(z;h) > C$ for some $C> 0$ by further shrinking $k$ if necessary.
\end{proof}

\begin{prop} \label{thm:pointdecay}
Let $T>0$. There exist constants $h_0 > 0,\, C>0$, and $c>0$ depending on $T$ such that
\[
\| \exp\left(\textstyle{\frac{z^2}{ch}}\right) \, u \|_{L^2(\Omega)} \leq C \left(\|u\|_{L^2(\Omega)} + h^{-1}\|\exp\left(\textstyle{\frac{z^2}{ch}}\right)(P^\sharp(h)-E(h))u\|_{L^2(\Omega)}\right),
\]
for all $h\in(0,h_0)$, $u\in \DOM^\sharp$ and $E(h)$ satisfying $E(h) < 1+T h$.
\end{prop}

\begin{proof}
Fix an arbitrary $\delta>0$. By Lemma $\ref{lem:parabola}$, we can choose $c>0$ so that if $\phi(z) = z^2/c$, then
\begin{equation} \label{eq:agmonpositivity}
V_0 + h^2 V_1 - (1+Th) - (\phi')^2 > \delta h,
\end{equation}
on $\Omega^+(1+(T+2\delta)h)$. Now fix a small $\varepsilon > 0$ and for ease of notation write 
\[
Z_1 = z'_{A}(T+2\delta)+\varepsilon, \quad Z_2 = z'_{A}(T+2\delta)+2\varepsilon
\]
Let $\eta$ be a smooth cutoff function with uniformly bounded derivative so that $\eta \equiv 0$ on $[0,Z_1]$ and $\eta \equiv 1$ on $[Z_2, h^{-1/2}z_{\mathrm{max}}]$. Set $\chi(z;h) = \eta(h^{-1/2}z;h)$. Then $\supp \chi$ is contained in $\Omega^+(1+(T+2\delta)h)$. Now apply Equation \eqref{eq:agmonintbyparts} with $f = e^{\phi/h}\chi u$ and \eqref{eq:agmoncommutator}. By the inequality \eqref{eq:agmonpositivity} and Lemma \ref{lem:weightedhardy},  along with Cauchy-Schwarz on the term involving $(P^\sharp(h)-E(h))u$, we obtain
\begin{multline*}
\delta h \, \| e^{\phi/h}\, \chi\, u\|^2_{L^2(\Omega)} \leq h^2 \left<u,\left((\chi')^2 + 2\, h^{-1}{\phi'}\, \chi'\, \chi \right)e^{2\phi/h}\,u\right>_{L^2(\Omega)}  \\ + \|e^{\phi/h}\chi(P^\sharp(h)-E(h))u\|_{L^2(\Omega)}\, \|e^{\phi/h}\chi u\|_{L^2(\Omega)} .
\end{multline*}
This inequality is of the form $\delta h p \leq r + p^{1/2}q^{1/2}$ which implies $\delta^2 h^2 p \leq 2\delta hr + q$. Thus
\begin{multline*}
\| e^{\phi/h}\, \chi\, u\|^2_{L^2(\Omega)} \leq 2\delta^{-1} h \left<u,\left((\chi')^2 + 2\, h^{-1}{\phi'}\, \chi'\, \chi \right)e^{2\phi/h}\,u\right>_{L^2(\Omega)}  \\ + (\delta h)^{-2} \|e^{\phi/h}(P^\sharp(h)-E(h))u\|^2_{L^2(\Omega)}.
\end{multline*} 
But
\[
\sup_{\Omega} |\chi'| = O(h^{-1/2}),
\]
and since $\supp \chi' = h^{1/2}[Z_1,Z_2]$, we see that
\[
\sup_{\supp \chi'} \exp(\phi/h) = O(1), \qquad \sup_{\supp \chi'} |\phi'| = O(h^{1/2}).
\]
Thus
\[
\| e^{\phi/h} \, u \|_{L^2(Z_2 h^{1/2}, z_\mathrm{max}]} \leq C_1 \|u\|_{L^2(\Omega)} + C_2 h^{-1} \|e^{\phi/h}(P^\sharp(h)-E(h))u\|_{L^2(\Omega)}.
\]
The final result now follows since  
\[
\|e^{\phi/h} \, u \|_{L^2(0,Z_2 h^{1/2}]} \leq C_3  \|u\|_{L^2(\Omega)}.
\]

\end{proof}

For the next proposition, fix an $S>0$. We will use the notation $\Sigma_i,\, i= 1,2$ to denote an interval of the form $\Sigma_i = (A_i, z_\mathrm{max}]$, where $z_A(1+S) < A_2 < A_1$. We then have $\Sigma_1 \Subset \Sigma_2 \Subset \Omega^+(1+S)$ with respect to the topology on $\Omega$. 

\begin{prop} \label{thm:hexpdecay}
Let $S > 0$ satisfy $1+S < V_0(z_\mathrm{max})$. There exist constants $h_0>0,\, C>0$, and $\epsilon>0$ depending on $S$, such that  
\[
\| u \|_{L^2(\Sigma_1)} \leq C \left(e^{-\epsilon/h} \|u\|_{L^2(\Sigma_2)}+\|(P^\sharp(h) - E(h))u \|_{L^2(\Sigma_2)}\right),
\]
for all $h\in (0,h_0)$, $u\in \DOM^\sharp$, and each $E(h)$ satisfying $E(h) < 1+S$.
\end{prop}

\begin{proof}
For $\delta>0$ small enough, we may assume that $\Sigma_2 \Subset \Omega^+(1+(S+2\delta))$. Choose a smooth cutoff  $\chi_1$ so that $\chi_1\equiv 1$ on $\Sigma_1$ and $\supp \chi_1 \subseteq \Sigma_2$. Then choose $\chi_2$ with $\chi_2\equiv 1$ on $\supp \chi_1$ and $\supp \chi_2 \subseteq \Sigma_2$. Then we can find $\epsilon$ such that if $\phi(z)= \epsilon \chi_1$ then 
\[
\delta < V_0 + h^2 V_1 - (1+S) - (\phi')^2.
\]
on $\Sigma_2$. 
Now proceed as in the previous proposition, again using Equations \eqref{eq:agmonintbyparts}, \eqref{eq:agmoncommutator}, Lemma \ref{lem:weightedhardy}, and Cauchy--Schwarz, to obtain
\begin{multline*}
\delta \|e^{\phi/h} \chi_2 u\|^2_{L^2(\Omega)} \leq h^2 \left<u,\left((\chi_2')^2 + 2\, h^{-1}\chi_1'\, \chi_2'\, \chi_2 \right)e^{2\phi/h}\,u\right>_{L^2(\Omega)}\\ + \|e^{\phi/h}\chi_2 (P^\sharp(h) - E(h))u \|_{L^2(\Omega)}\, \|e^{\phi/h} \chi_2 u\|_{L^2(\Omega)}.
\end{multline*}
Arguing as in the previous proposition and using that $\chi_1 \equiv 0$ on $\supp \chi_2'$, we get that 
\[
e^{\epsilon/h}\, \| u\|_{L^2(\Sigma_1)} \leq C_1 h \|u\|_{L^2(\Sigma_2)} + C_2 e^{\epsilon/h} \| (P^\sharp(h) - E(h))u \|_{L^2(\Sigma_2)}.
\]
Multiplying through by $e^{-\epsilon/h}$ gives the desired result.
\end{proof}

We can combine this result with a standard rescaled elliptic estimate \cite[Chapter 7]{Zworski:2012}, using the Dirichlet boundary condition at $z= z_\mathrm{max}$.

\begin{cor} \label{cor:sobolev}
With the same hypotheses as above,
\[
\| u \|_{H^2_h(\Sigma_1)} \leq C \left(e^{-\epsilon/h} \|u\|_{L^2(\Sigma_2)}+\|(P^\sharp(h) - E(h))u \|_{L^2(\Sigma_2)}\right).
\]
\end{cor}
The norm on $H^k_h(U)$ is given by $\|u\|^2_{H^k_h} = \sum_{|\alpha|\leq k} \int_U |(hD)^\alpha)u|^2 dx$.

\subsection{Asymptotic expansion for low lying quasimodes} \label{sect:expansion}
Before constructing quasimodes for $P(h)$, we apply the results of the previous section to obtain asymptotic expansions for the lowest eigenvalues of $P^\sharp(h)$.

\begin{prop} \label{thm:h1/2}
Let $T>0$. There exists $h_0>0$ depending on $T$ so that for all $h\in(0,h_0)$ there is a one-to-one correspondence between the numbers $\widetilde{E}_n(h) = 1+2(2n+1+\nu)h$ and the eigenvalues $E^\sharp_n(h)$ of $P^\sharp(h)$ which are both less than $1+Th$. Moreover, there are constants $C_n>0$ so that
\[
|E^\sharp_n(h) - \widetilde{E}_n(h)| < C_n h^{3/2}.
\]
\end{prop}
\begin{proof}

Fix some $1<S<V_0(z_\mathrm{max})$ and note that $1+Th < 1+S$ for $h$ small enough. Fix $A > z_A(1+S)$ and let $\chi$ be a smooth compactly supported function with $\chi \equiv 1$ on $(0,A]$ and $\supp \chi = \Omega$.

First, let $\widetilde{E}(h) < 1+Th$ be an eigenvalue of $\widetilde{P}(h)$ with normalized eigenvector $\widetilde{u}(h)$. Then $\chi \widetilde{u}(h) \in \DOM^\sharp$.   We compute
\[
(P^\sharp(h) - \TE(h))(\chi\,\widetilde{u}(h)) = \chi R(h) \widetilde{u}(h) + \left[-h^2 \textstyle{\DDZ},\chi \right] \widetilde{u}(h).
\]
By their explicit forms both $\widetilde{u}(h)$ and its derivative are exponentially decaying with Gaussian weight $-z^2/2h$. Since $R(h)=O(z^3)+O(h^2)$ on $\Omega$, we get   
\[
\| \chi R(h) \widetilde{u}(h) \|_{L^2(\Omega)} = O(h^{3/2}),
\]
and
\[
\| \left[-h^2 \textstyle{\DDZ},\chi \right] \widetilde{u}(h) \|_{L^2(\Omega)} = O(e^{-\epsilon/h}).
\]
The constants in the $O$-terms are uniform in $\TE(h)<1+Th$. Thus 
\[
\|(P^\sharp(h) - \TE(h))(\chi\, \widetilde{u}(h)) \|_{L^2(\Omega)} = O(h^{3/2}).
\] 
Moreover since $\|\chi\, \widetilde{u}(h)\|_{L^2(\Omega)} = 1- O(e^{-d/h})$, where the $O$-term is uniform for $\TE(h)<1+Th$, it follows that $\chi\, \widetilde{u}(h)$ can be normalized without affecting the $O(h^{3/2})$ bound. The spectral theorem then guarantees the existence of an eigenvalue $E^\sharp(h)$ for $P^\sharp(h)$ satisfying $|E^\sharp(h)-\widetilde{E}(h)| < C h^{3/2}$.

For the other direction, suppose $u^\sharp(h)$ is a normalized eigenvector with eigenvalue $E^\sharp(h)$. Then $\chi u^\sharp(h)\in\widetilde{\DOM}$ if we extend it by zero outside of $\Omega$. As above, compute 
\[
(\widetilde{P}(h)-E^\sharp(h))(\chi\, u^\sharp(h)) = - \chi R(h) u^\sharp(h) + \left[-h^2 \textstyle{\DDZ},\chi \right] u^\sharp(h).
\]
This time we apply Proposition \ref{thm:pointdecay} and use the fact that $P^\sharp(h)u^\sharp(h) = E^\sharp(h)u^\sharp(h)$ to conclude that    
\[
\| \chi R(h) u^\sharp(h) \|_{L^2(0,\infty)} = O(h^{3/2}),
\]
and use Corollary \ref{cor:sobolev} to see that
\[
\| \left[-h^2 \textstyle{\DDZ},\chi \right] u^\sharp(h) \|_{L^2(0,\infty)} = O(e^{-\epsilon/h}),
\]
where the constants in the $O$-terms are uniform in $E^\sharp(h)<1+Th$. By another application of Proposition \ref{thm:hexpdecay}, we write
\[
\| \chi\, u^\sharp(h) \|^2_{L^2(0,\infty)} = \| u^\sharp(h) \|^2_{L^2(\Omega)} - \left< (1-\chi^2) u^\sharp(h),u^\sharp(h) \right >_{L^2(\Omega)} = 1 - O(e^{-\epsilon/h}), 
\] 
where the $O$-term is again uniform for $E^\sharp(h) < 1+ Th$. It follows that $\chi\, u^\sharp(h)$ can be normalized as above and the spectral theorem then guarantees the existence of an eigenvalue $\TE(h)$ for $\widetilde{P}(h)$ satisfying $|E^\sharp(h)-\widetilde{E}(h)| < C h^{3/2}$.
\end{proof}

\begin{cor} For each $\delta>0$ there exists $h_0>0$ such that $E^\sharp_0(h) \geq 1+(2-\delta)h$ for all $h\in(0,h_0)$.
\end{cor}

Next, we improve on the previous result by producing a full asymptotic expansion for the $E^\sharp(h)$ lying close to $E=1$. We refer to \cite[Chapter 12]{Helffer:2005} for the usual case of a nondegenerate potential well in $\mathbb{R}^n$. Let $U(h):L^2(\Omega)\rightarrow L^2(h^{-1/2}\Omega)$ denote the unitary dilation $(U(h)u)(x) = h^{1/4}u(h^{1/2}x)$. Then $U(h) \widetilde{P}(h) U(h)^{-1} = 1 + h Q_0$ where $Q_0 = \widetilde{P}(1)-1$, in other words $h$ scales out exactly. Keeping this in mind, we conjugate $P^\sharp(h)$ by $U(h)$ and collect like powers of $h^{1/2}$ in the Laurent series of $V(h^{1/2}x;h)$ to formally write  
\[
U(h) P^\sharp(h) U(h)^{-1} = 1+h\sum_{k=0}^{\infty}h^{k/2} Q_k
\]
where $Q_k$ for $k\geq 1$ is a polynomial of degree at most $k+2$ (whose coefficients are independent of $h$). Before proceeding with the construction, we remark that the same methods as in Propositions \ref{thm:pointdecay} and \ref{thm:hexpdecay} give the following, which we state as lemmas.

\begin{lem} \label{lem:Qpointdecay}
Let $T>0$. There exist constants $C>0$, and $c>0$ depending on $T$ such that
\[
\| \exp\left(\textstyle{\frac{x^2}{c}}\right) \, u \|_{L^2(0,\infty)} \leq C \left(\|u\|_{L^2(0,\infty)} + \|\exp\left(\textstyle{\frac{x^2}{c}}\right)(Q_0-E)u\|_{L^2(0,\infty)}\right),
\]
for all $u\in \widetilde{\DOM}$ and $E$ satisfying $E<T$.
\end{lem}
The proof of this fact goes through as before with $h=1$; the only difference is that since we now have an unbounded interval, we need to work with the bounded weight $\phi_\alpha = \frac{\phi}{1+\alpha \phi}$ and then justify the limit as $\alpha\rightarrow 0$.
\begin{lem} \label{lem:Qhexpdecay}
Let $S > 0$. There exist constants $h_0>0,\, C>0,\, \epsilon>0$ depending on $S$, such that for any fixed interval $\Sigma_1 \Subset \left\{z:\, z^2 > S\right\}$,  
\[
\| u \|_{H_h^2(\Sigma_1)} \leq C \left(e^{-\epsilon/h} \|u\|_{L^2(\Sigma_2)}+\|(\widetilde{P}(h) - E(h))u \|_{L^2(\Sigma_2)}\right),
\]
for $h\in (0,h_0)$, $u\in \widetilde{\DOM}$, and all $E(h) < 1+S$, whenever $\Sigma_1\Subset \Sigma_2 \Subset (0,\infty)$.
\end{lem}
Here the proof goes through unchanged. Note that in the proof of both of these results, we should use the ordinary Hardy inequality in place of Lemma \ref{lem:weightedhardy}.

\begin{prop} \label{thm:expansion}
Fix $n\geq0$. There exists $h_0>0$ depending on $n$ such that $E^\sharp_n(h)$ has an asymptotic expansion
\[
E^\sharp_n(h) = 1+2(2n+1+\nu)h + \sum_{k=1}^{N}E_{n,k} h^{\frac{k+2}{2}} + O(h^{\frac{N+3}{2}}).
\]
for $h\in (0,h_0)$.
\end{prop}
\begin{proof}
Start with an eigenvector $v_{n,0} = \widetilde{u}_n(1)$ of $Q_0$ with eigenvalue $E_{n,0} = 2(2n + 1 + \nu)$. We are interested in formally solving 
\[
\left(\sum_{k=0}^\infty h^{k/2}\left(Q_k-E_{n,k}\right)\right)\left(\sum_{k=0}^\infty h^{k/2}v_{n,k}\right) \sim 0,
\]
where we need to find the $E_{n,k}$ and $v_{n,k}$ for $k\geq 1$. Expanding the product above and collecting like powers of $h^{1/2}$, we find the sequence of equations,
\[ 
(Q_0 - E_{n,0})v_{n,k} = - \sum_{r=0}^{k-1} (Q_{k-r}-E_{n,k-r})v_{n,r}.
\]
By Fredholm theory, we can solve this equation for $v_{n,k}$ as soon as the right hand side is orthogonal in $L^2(0,\infty)$ to the kernel of $Q_0 - E_{n,0}$, namely $\mathrm{span}\,(v_{n,0})$. We can inductively impose the Fredholm condition by setting
\[
E_{n,k} = \sum_{r=1}^{k-1} \left <(Q_{k-r}-E_{n,k-r})v_{n,r}, v_{n,0} \right> + \left < Q_k v_{n,0},v_{n,0}\right >,
\]
once $E_{n,j}$ and $v_{n,j}$ have been determined for $0\leq j \leq k-1$. Now let $\chi$ be the same cutoff function as in Proposition \ref{thm:h1/2}, and set 
\[ 
w_{n,N}(z;h) = \sum_{k=0}^N h^{k/2} v_{n,k}(h^{-1/2}z).
\] 
We wish to show that
\[
\left\lVert \left(P^\sharp(h) - \left(1+h \sum_{k=0}^N h^{k/2} E_{n,k} \right)\right)\chi(z) w_{n,N}(h)\right\rVert_{L^2(\Omega)} = O(h^{\frac{N+3}{2}}). 
\]
The proof proceeds as before by commuting the operator with $\chi$ at the loss of a commutator term. We are then left with estimating two terms: first the $L^2(h^{-1/2}\Omega)$ norm of $h^{\frac{N+3}{2}} \chi(h^{1/2}x) R_N(x) w_{n,N}(x;1)$ where $R_N$ is polynomially bounded. Then we need to estimate the $H_h^1(\supp \chi')$ norm of $w_{n,N}(h)$. Since $v_{n,0}$ is exponentially decaying with weight $-x^2/2$, and since each term in $w_{n,N}$ now solves an \emph{inhomogeneous} equation, we use Lemmas \ref{lem:Qpointdecay} and \ref{lem:Qhexpdecay} to inductively obtain the necessary decay of $w_{n,N}$. Similarly we can show that $w_{n,N}$ is normalizable and by the spectral theorem there is an eigenvalue of $P^\sharp(h)$ 
such that the distance to $1+h \sum_{k=0}^N h^{k/2} E_{n,k}$ is of order $O(h^{\frac{N+3}{2}})$. This eigenvalue must be $E^\sharp_n(h)$ since the lowest eigenvalues of $P^\sharp(h)$ are separated at a distance greater than $Ch$.       
\end{proof}

\begin{rem}
In the case of a nondegenerate potential well on $\mathbb{R}$, only integral powers of $h$ occur in the expansion of the lowest eigenvalues. This is in contrast to the situation here. Consider for example when $d=3$. In that case the Laurent expansion of $V$ is
\[
V(z;h) = 1 + h^2(\nu^2-1/4)z^{-2} + z^2 -\mu z^3 + \ldots
\]
and so 
\[
E_{n,1} = \int_0^\infty -\mu x^3\, \widetilde{u}_n(1)^2 dx,
\]
which is nonvanishing. Of course we are actually interested in an expansion of $\omega^\sharp_{n,\ell} = \sqrt{E^\sharp_n(h)}$ --- this expansion occurs in half-powers of $\ell^{-1}$. In Section \ref{sect:originalparameters} we examine the vanishing of certain coefficients depending on the dimension. In particular, we address a conjecture of Dias et al. \cite{Dias:2012} on the behavior of these coefficients in dimensions $d=3,4,5$.  
\end{rem}
 
\subsection{Construction of quasimodes} \label{sect:quasimodes}

In this section we present the main theorem on the existence of exponentially accurate quasimodes for $P(h)$.

\begin{theo} \label{thm:quasimodes}
Let $S > 0$ satisfy $1+S < V_0(z_\mathrm{max})$. There exists 
\begin{itemize}
	\item Constants $h_0>0,\, D_1,D_2>0$ depending on $S$ and an integer valued function $m(h)\geq 1$.
	\item Real numbers $\left\{E^\sharp_n(h)\right\}_{n=0}^{m(h)}$ with the property that $1 <E^\sharp_n(h)< 1+S$ for $h\in(0,h_0)$.
	\item Smooth functions $\left\{u_n(h)\right\}_{n=0}^{m(h)}\subset \DOM(h)$ with $\|u_n(h)\|_{L^2(0,\infty)}=1$, all supported in a compact set $K$.
\end{itemize}
such that for all $h\in (0,h_0)$, the functions $u_n(h)$ satisfy

\begin{enumerate} \itemsep8pt
	\item $\|\left(P(h)-E^\sharp_n(h)\right)u_n(h) \|_{L^2(0,\infty)} \leq e^{-D_1/h}$,
	\item $|\left\langle u_i(h),u_j(h) \right\rangle -\delta_{ij}| \leq e^{-D_2/h}.$ 
\end{enumerate}
\end{theo}
\begin{proof}

Define $m(h)$ to be the number of $E^\sharp_n(h)$ satisfying $E^\sharp_n(h) < 1+S$.  Fix $A > z_A(1+S)$ and let $\chi$ be a smooth compactly supported function with $\chi \equiv 1$ on $(0,A]$ and $\supp \chi = \Omega$. Set $u_n(h) = \chi\, u^\sharp_n(h)$ for $n\in\left\{0,1,\ldots,m(h)\right\}$ so that $u_n(h)\in\DOM$ if we extend it by zero outside of $\Omega$. Then compute
\[
\|(P(h)-E_{n}(h))u_i(h)\|_{L^2(0,\infty)} = \left\lVert\left[-h^2 \textstyle{\DDZ},\chi \right] u^\sharp_i(h)\right\rVert_{L^2(0,\infty)} \leq e^{-D_1/h}
\]
by Corollary \ref{cor:sobolev}. Since the $u_n(h)$ can be normalized the first claim follows. As for the second claim, simply write $u_n(h) = u^\sharp_n(h) + (\chi-1)u^\sharp_n(h)$ where of course we mean the extension of $u^\sharp_n(h)$ by zero outside $\Omega$. Since $\|(\chi-1)u^\sharp_n(h)\|_{L^2(0,\infty)}=O(e^{-D_2/h})$ by shrinking the support of $\chi$ if necessary, we see that $\left\langle u_i(h), u_j(h) \right\rangle = O(e^{-D_2/h})$ for $i\neq j$.
\end{proof}

\section{Existence of resonances}

\subsection{Black box model}
To define the resonances of $P(h)$, we first give a formulation in terms of \emph{black box scattering}. It is important to note that all of the results in this section were first obtained for elliptic operators with coefficients that are dilation analytic at infinity \cite{Sjostrand:1991}, \cite{Sjostrand:1996}, \cite{Tang:1998}, and are all applicable to the problem at hand. The presentation we give here is an alternative based on exponential decay of the potential rather than analyticity \cite{Gannot1}. 

This framework is useful based on the following observation: outside any ball containing the origin, it is $V$ that is exponentially decaying, not $W$ in general. If $W$ \emph{was} exponentially decaying, we could view $L_\nu(h)$ as the ``free'' operator and write $P(h) = L_\nu(h) + W$. The (weighted) resolvent of $L_\nu(h)$ has an explicit integral kernel and continues analytically to a strip in the lower half-plane with favorable norm estimates. It would then be standard to meromorphically continue the (weighted) resolvent of $P(h)$ in terms of the resolvent of $L_\nu(h)$, see for example \cite{Simon:2000}. Since this is not the case, the black box model we now present allows us to circumvent this issue.

Let $Y$ denote either $Y=\mathbb{R}^{n}$ or $Y=(0,\infty)$ and suppose $\mathcal{H}$ is a Hilbert space with an orthogonal decomposition $\mathcal{H} = \mathcal{H}_{R_0} \oplus L^2(Y \backslash B(0,R_0))$ where $B(0,R_0) = \left\{y\in Y: \, |y| < R_0\right\}$. The orthogonal projections onto $\mathcal{H}_{R_0}$ and $L^2(Y\backslash B(0,R_0))$ will be denoted $1_{B(0,R_0)}u = u|_{B(0,R_0)}$ and $1_{Y\backslash B(0,R)}u = u|_{Y\backslash B(0,R_0)}$ for $u\in \mathcal{H}$. 

Suppose $P(h)$ is an unbounded self-adjoint operator on a domain $\DOM \subset \mathcal{H}$. We say that $P(h)$ satisfies the black box hypotheses if the following hold:
   
\begin{enumerate} \itemsep8pt
	\item $1_{Y\backslash B(0,R_0)} \DOM = H^2_h(Y\backslash B(0,R_0))$, and conversely if $u\in \DOM$ vanishes near $B(0,R_0)$ then $u\in H^2_h(Y\backslash B(0,R_0)$
	\item $1_{B(0,R_0)} (P(h)+i)^{-1}: \mathcal{H} \rightarrow \mathcal{H}_{R_0}$ is compact.
	\item There exists a symmetric real-valued matrix and a real-valued function 
\[
a_{ij}(y;h) \in C_b^\infty(Y\backslash B(0,R_0)), \quad V(y;h) \in C_b^\infty(Y\backslash B(0,R_0))
\] 
with all derivatives uniformly bounded in $h$, so that
\[
(P(h)u)|_{Y\backslash B(0,R_0)} = (-h^2 \sum_{i,j} \partial_{i} a_{ij} \partial_{j} + V)(u|_{Y\backslash B(0,R_0)}), \, u\in \DOM.
\]
	\item The metric coefficients $(a_{ij})$ are uniformly elliptic.
	\item The perturbation decays exponentially to the Laplacian in the sense that there exists $\tau>0, \delta>0$ so that
\[
|a_{ij}(y;h)-\delta_{ij}| \leq C e^{-(2\tau+\delta) |y|}, \quad |V(y;h)| \leq C e^{-(2\tau+\delta) |y|},\, y\in Y\backslash B(0,R_0).
\]
\end{enumerate}
A parametrix construction and analytic Fredholm theory gives the meromorphic continuation: 
\begin{prop} [{\cite[Proposition 1.5]{Gannot1}}]
The resolvent $R(E;h) = (P(h) - E)^{-1}$, analytic in the upper half-plane, admits a meromorphic continuation across $(0,\infty)$ to the strip $\left\{\Re E > 0\right\} \cap \left \{\Im E > -\tau h \right\}$ as a bounded operator from $e^{-\tau |y|}\mathcal{H}$ to $e^{\tau |y|}\mathcal{H}$. 
\end{prop}
Here we define $e^{\pm \tau |y|}\mathcal{H} = \mathcal{H}_{R_0} \oplus e^{\pm \tau |y|} L^2(Y\backslash B(0,R_0))$. The set of resonances of $P(h)$ in this strip will be denoted by $\mathrm{Res}\,P(h)$ and a typical element will be denoted by $r(h)$. Under these hypotheses, the existence of localized quasimodes implies the existence of resonances rapidly converging to the real axis. This follows from an a priori bound on the continued resolvent of $P(h)$ away from resonances: let $\Gamma = (a,b) + i (c, -(\tau - \epsilon)h)$ where $0<a<b$ and $c>0,\epsilon>0$. Then there exists some $p>0$ such that 
\begin{equation} \label{eq:resolventbound}
\| R(E;h) \| \leq \exp(Ah^{-p} \log(1/g(h))),\quad E\in \Gamma \setminus \bigcup_{r(h)\in \mathrm{Res}\,P(h)} B(r(h),g(h)).
\end{equation}
Here the operator norm is taken between the exponentially weighted spaces above. If there did not exist resonances close to the real axis, then by a version of the three-lines lemma (often referred to in this context as the ``semiclassical maximum principle'' \cite{Tang:1998}), we could interpolate this bound in the lower half-plane with the self-adjoint bound $\| R(E,h) \| \leq C |\Im E|^{-1}$ in the upper half-plane to deduce a polynomial bound on the resolvent on the real axis. But such a bound would contradict the existence of a sufficiently accurate quasimode. More precisely, in the case of an exponentially decaying potential, \cite[Theorem 3]{Stefanov:2005} continues to hold:

\begin{theo} \label{thm:existenceofres}
Let $P(h)$ satisfy the black box hypotheses. Let $0<a_0<a(h)<b(h)<b_0<\infty$. Assume there is an $h_0>0$ such that for $h\in(0,h_0)$ there exists $m(h)\in\left\{1,2,\ldots\right\},\ E^\sharp_n(h)\in \left[a(h),b(h)\right]$, and $u_n(h)\in\DOM$ with $\|u_n(h)\|=1$ for $1\leq n\leq m(h)$ such that $\supp u_n(h) \subset K$ for a compact set $K$ independent of $h$. Suppose further that
\begin{enumerate} \itemsep8pt
	\item $\|(P(h)-E^\sharp_n(h))u_n(h)\| \leq R(h)$,
	\item Whenever a collection $\left\{v_n(h)\right\}_{n=1}^{m(h)}\subset \mathcal{H}$ satisfies $\|u_n(h)-v_n(h)\| < h^N/M$, then $\left\{v_n(h)\right\}_{n=1}^{m(h)}$ are linearly independent,
\end{enumerate}
	where $R(h) \leq h^{p+N+1}/C\log(1/h)$ and $C\gg 1,\, N\geq 0,\, M>0$. Then there exists $C_0>0$ depending on $a_0,b_0$ and the operator $P(h)$ such that for $B>0$ there exists $h_1<h_0$ depending on $A,B,M,N$ so that the following holds: Whenever $h\in (0,h_1)$, the operator $P(h)$ has at least $m(h)$ resonances in the strip
\[
	\left[a(h)-c(h)\log\frac1h, b(h)+c(h)\log\frac1h\right] - i\left[0,c(h)\right]
\]
	where $c(h) = \max(C_0BMR(h)h^{-p-N-1}, e^{-B/h})$.
\end{theo}

To prove \eqref{eq:resolventbound}, we construct an associated reference operator $P^\sharp(h)$ with discrete spectrum such that $(P(h)-E) \chi = (P^\sharp(h) -E) \chi$ where $\chi \equiv 1$ near $B(0,R_0)$. We then add as an additional hypothesis that the number of eigenvalues in each interval $[-L,L]$ with $L\geq 1$ satisfies 
\begin{equation} \label{eq:eigenvaluegrowth}
N(P^\sharp(h),[-L,L])\leq C(L/h^2)^{n^\sharp/2}.
\end{equation}
This allows us to estimate the singular values of $(P(h) -E)^{-1}\chi$ in terms of \eqref{eq:eigenvaluegrowth}, which is the main ingredient in the proof of \eqref{eq:resolventbound}; in fact, the number $p>0$ appearing in the exponential bound above is related to $n^\sharp$. For our purposes, we can construct $P^\sharp(h)$ by restricting $P(h)$ to a ball $B(0,R_1)$ and imposing a Dirichlet condition on $\partial B(0, R_1)$, where $R_1 \gg R_0$.

\subsection{Schwarzschild--AdS problem in the black box framework} We now apply the above formalism to our situation. As our Hilbert space we take 
\[
\mathcal{H} = L^2(0,\infty) = L^2(0,R_0) \oplus L^2(R_0,\infty)
\]
for some $R_0 \ll z_{\mathrm{max},0}$. Our operator will be $P(h)$ on $\DOM$ and we may take $P^\sharp(h)$ on $\DOM^\sharp$ as our reference operator. However, we do need to verify that the eigenvalues of $P^\sharp(h)$ satisfy \eqref{eq:eigenvaluegrowth}, in this case with $n^\sharp = 1$. 

\begin{prop} \label{prop:weyllaw}
There exists $h_0>0$ and $C>0$ such that for any $L\geq 1$ the number of eigenvalues of $P^\sharp(h)$ in $[-L,L]$ satisfies $N(P^\sharp(h),[-L,L]) < C(L^{1/2}/h)$ when $h\in (0,h_0)$.
\end{prop}
\begin{proof}
By Lemma \ref{lem:W(h)} we have $P^\sharp(h) \geq L^\sharp_\nu(h)$ and hence by the max-min principle, $N(P^\sharp(h),[-L,L]) \leq N(L^\sharp_\nu(h),[-L,L])$. The eigenvalue problem for $L^\sharp_\nu(h)$ is
\[
-h^2u''(z) + h^2\textstyle{\frac{\nu^2-\frac14}{z^2}}u(z) = k u(z), \quad \lim_{z\rightarrow 0} z^{\nu-1/2}u(z)=0,\quad u(z_{\mathrm{max},0})=0.
\]
The eigenvalues of $L^\sharp_\nu(h)$ are given by $k_n = \left(\frac{h}{z_{\mathrm{max},0}}\right)^2j_{\nu,n}^2$ where $j_{\nu,n}$ are the zeros of the first Bessel function $J_\nu$. The $j_{\nu,n}$ satisfy
\[
j_{\nu,n} = \left(n+\frac12\nu-\frac14\right)\pi + O(n^{-1})
\]
as $n\rightarrow\infty$. It follows that $N(L^\sharp_\nu(h),[-L,L]) = h^{-1}\left(\pi \sqrt{z_{\mathrm{max},0} L} + O(h)\right)$. The result thus follows with $C$ any constant larger than $\pi\sqrt{z_{\mathrm{max},0}}$. 
\end{proof}

\begin{prop} \label{prop:relativecompactness} The Schwarzschild--AdS problem satisfies the black box hypotheses.
\end{prop}
\begin{proof}
The only fact that needs checking is the compactness of $1_{B(0,R_0)} (P(h)+i)^{-1}$. We view $1_{B(0,R_0)}$ as multiplication by an indicator function on $\mathcal{H}$ and hence interpret $1_{B(0,R_0)} (P(h)+i)^{-1}$ as a bounded operator on $\mathcal{H}$. We use the following fact: any operator on $L^2(0,\infty)$ of the form $f(x)g(\sqrt{L_\nu(h)})$, where $f,g\in L^2(0,\infty)$, is Hilbert--Schmidt, see \cite[Proposition 2.7] {Simon:2000}. The proof relies on the fact that the Hankel transform gives an eigenfunction expansion for $L_\nu(h)$; this fact is classical for $\nu \geq 1$, while for the case $0< \nu < 1$ (and for our choice of boundary condition at $z=0$) we refer to \cite{Everitt:2007}.  Let $g=(y^2 + i)^{-1}$ so that $(L_\nu(h)+i)^{-1} = g(\sqrt{L_\nu(h)})$ and $g\in L^2(0,\infty)$. Then
\begin{multline*}
1_{B(0,R_0)} (L_\nu(h)+W+i)^{-1} = 1_{B(0,R_0)} (L_\nu(h)+i)^{-1} \\
- 1_{B(0,R_0)} (L_\nu(h)+W+i)^{-1} W (L_\nu(h)+i)^{-1}.
\end{multline*}
Both summands on the right hand side are Hilbert--Schmidt first by choosing $f=1_{B(0,R_0)}$ and then $f=W$.
\end{proof}

We finally come to our theorem on the existence of resonances with exponentially small imaginary parts.

\begin{theo} \label{thm:resonances}
Assume the hypotheses and notations of Theorem \ref{thm:quasimodes}. There exists $h_1>0$ and $D_0>0$ depending on $S$ such that for all $h\in (0,h_1)$ there is a one-to-one corresponce between $\sigma(P^\sharp(h))\cap [1,1+S]$ and $\mathrm{Res}\,P(h) \cap [1, 1+S+e^{-D_0/h}]-i[0,e^{-D_0/h}]$. Moreover, for each quasimode $E^\sharp_n(h)$ there is a corresponding resonance $r_n(h)$ with $|E^\sharp_n(h)-r_n(h)|\leq e^{-D_0/h}$. In particular there are exactly $m(h)$ resonances in $[1, 1+S+e^{-D_0/h}]-i[0,e^{-D_0/h}]$.
\end{theo}
\begin{proof}
For the energy interval take $[a_0,b_0]=[1,1+S]$. Choose $C_0$ such that $c(h)\log\frac{1}{h} \leq e^{-C_0/h}$ in the notation of Theorem \ref{thm:existenceofres}. For each quasimode $E^\sharp_n(h)$ consider the boxes 
\begin{align*}
\Omega_n &= [E^\sharp_n(h)-2e^{-C_0/h}, E^\sharp_n(h)+2e^{-C_0/h}],\\ 
\Omega'_n &= [E^\sharp_n(h)-4e^{-C_0/h}, E^\sharp_n(h)+4e^{-C_0/h}]. 
\end{align*}
We now group together those $\Omega'_n$ which are not disjoint into $J(h)=O(h^{-1})$ clusters and let $[a_j(h),b_j(h)]$ denote the smallest connected interval containing the corresponding $\Omega_n$. Since $m(h) = O(h^{-1})$, the width of $[a_j(h)-e^{-C_0/h},b_j(h)+e^{-C_0/h}]$ is less than $Ch^{-1}e^{-C_0/h}$. Moreover the distance between any two boxes $[a_j(h),b_j(h)]$ and $[a_i(h),b_i(h)]$ is greater than $4e^{-C_0/h}$, which implies that the resonances in $[a_j(h)-c(h)\log\frac{1}{h},b_j(h)+c(h)\log\frac{1}{h}]$ and $[a_i(h)-c(h)\log\frac{1}{h},b_i(h)+c(h)\log\frac{1}{h}]$ are all disjoint. We now apply Theorem \ref{thm:existenceofres} to each box $[a_j(h),b_j(h)]$ to conclude that there are at least $m_j(h)$ resonances in $[a_j(h)-c(h)\log\frac{1}{h},b_j(h)+c(h)\log\frac{1}{h}]-i[0,c(h)]$, where $m_j(h)$ is the number of quasimodes in $[a_j(h),b_j(h)]$. Since the width of each box is exponentially small, we see that to quasimode $E^\sharp_n(h)$ we can associate a unique resonance $r_n(h)$ satisfying $|E^\sharp_n(h)-r_n(h)|\leq e^{-D_0/h}$ with a uniform constant $D_0$.

The converse follows as in the proof of \cite[Lemmas 4.5, 4.6]{Nakamura:2002} where it is shown that each resonant state is exponentially small inside the barrier and hence can be truncated to produce a quasimode.
\end{proof}

\subsection{Restoring the original parameters} \label{sect:originalparameters}

We now restate our results in terms of the angular momentum $\ell$ and the original spectral parameter $\omega$. The corresponding quasimodes and resonances will be denoted by
\begin{align*}
\omega^\sharp_{n,\ell} &= (\ell - 1 + d/2) E^\sharp_n\left(\left(\ell - 1 + d/2\right)^{-1}\right)^{1/2},\\
\omega_{n,\ell} &= (\ell - 1 + d/2) r_n\left(\left(\ell - 1 + d/2\right)^{-1}\right)^{1/2}.
\end{align*}
The asymptotic expansion for the low lying quasimodes (and hence for the real parts of the corresponding resonances) then takes the form

\begin{equation} \label{eq:ellexpansion}
\omega^\sharp_{n,\ell} = \ell + (2n+\nu +d/2) + c_{n,1} \ell^{-1/2} + c_{n,2} \ell^{-1} +\ldots.
\end{equation}

The two term approximation $\ell + (2n+\nu +d/2)$ was already proposed in \cite{Festuccia:2008}. In a recent work, Dias et al.\cite{Dias:2012} numerically analyzed the difference $\Re \omega_{n,\ell} - (\ell + 2n+\nu +d/2)$. By fitting to a power law, they found the difference behaves as $\ell^{-\frac{d-2}{2}}$.  In light of our asymptotic expansion, this at first seems surprising --- it implies that in dimension $d$, the process of taking a square root to pass from $E_n(h)$ to $\omega_{n,\ell}$ annihilates all the coefficients $c_{n,1},\ldots, c_{n,d-3}$. This seems more plausible when one takes into account how the asymptotic expansion is constructed: the first nonzero coefficient $E_{n,k}$ in the expansion occurs precisely at that first value of $k\geq 1$ so that that $Q_k$ is nonzero. In dimension $d$, this value of $k$ is not equal to $d$. However, viewing the equation in the original $r$-coordinate, we recall that
\[
V(r;h) = 1+h^2\left(\nu^2-\frac{1}{4}\right)r^2 + \frac{1}{r^2} -\frac{\mu}{r^d} + \textrm{lower order terms}.
\]
Since $z(r) \sim \frac{1}{r}$ as $z\rightarrow 0$, we see that to leading order $1+h^2(\nu^2-\frac{1}{4})r^2 + \frac{1}{r^2}$ corresponds to $1+h^2(\nu^2-\frac{1}{4})z^{-2} + z^2$ and hence $-\frac{\mu}{r^d}$ can be thought of as the first perturbative term. Roughly speaking, the first perturbative term \emph{is} of the size $\frac{1}{r^d} \sim z^d$ in the $r$-coordinate. The issue is that when passing to the Regge-Wheeler coordinate, $-\mu z^d$ is no longer the first perturbative term owing to lower order terms in the expansion $r(z) = \frac{1}{z} + \ldots$.  The question is then whether one can simply run the argument in the $r$-coordinate, but in that case we no longer have a well understood model operator like $\widetilde{P}(h)$.

Nevertheless, we can establish the following result for small dimensions:

\begin{prop} In dimensions $d=3,4$, the first nonvanishing coefficient in the expansion of $\Re \omega_{n,\ell}$ is $c_{n,d-2}$. When $d\geq 5$, both $c_{n,1}$ and $c_{n,2}$ vanish.
\end{prop}
\begin{proof}
By explicitly calculating the Laurent expansion of $r(z)$, we have

\begin{enumerate}
	\item When $d=3$: $Q_1(x) = -\mu x^3$.
	\item When $d=4$: $Q_1(x) = 0,\, Q_2(x) = \frac{\nu^2-1}{3} + \left(\frac{2}{3}-\mu\right)x^4$.
	\item When $d \geq 5$: $Q_1(x) = 0,\, Q_2(x) = \frac{\nu^2-1}{3} + \frac{2}{3}x^4$.
\end{enumerate}

In these dimensions we are only concerned with the two coefficients $c_{n,1}, c_{n,2}$ and these are readily obtained from the $E_{n,k}$ by
\begin{align*}
c_{n,1} &= \frac{E_{n,1}}{2}, \\
c_{n,2} &= \left(\frac{E_{n,2}}{2}-\frac{E_{n,0}^2}{8}\right).
\end{align*}   

We also see that 

\begin{enumerate}
	\item When $d=3$: $E_{n,1} = \left <Q_1 \widetilde{u}_n(1),\widetilde{u}_n(1)\right >$. 
	\item When $d\geq4$: $E_{n,1} = 0,\, E_{n,2} = \left < Q_2 \widetilde{u}_n(1),\widetilde{u}_n(1)\right >$.
\end{enumerate}
            
The inner products are of course taken in $L^2(0,\infty)$. At this stage, we remark that when $d=3$ we clearly have $E_{n,1}\neq 0$ and hence $c_{n,1}\neq 0$; we say no more about this case. In the other dimensions we need to actually compute the matrix elements: a general expression can be found in \cite{Hall:2000},
\[
\left < x^\alpha \widetilde{u}_n(1),\widetilde{u}_n(1) \right > = \frac{\Gamma(n+1+\nu)}{n!\,\Gamma(1+\nu)^2}\sum_{r=0}^{n}\Gamma\left(\frac{\alpha}{2}+1+\nu+r\right)\frac{\left(-\frac{\alpha}{2}-r\right)_n}{(1+\nu)_n}\frac{(-n)_r}{(1+\nu)_r}\frac{1}{r!},
\]
where $(a)_r = \Gamma(a+r)/\Gamma(a)$ unless $a=-m$ is a negative integer, in which case $(-m)_r=(-m)(-m+1)\cdots(-m+r-1)$. When $\alpha = 4$, the quantity
$\left(-\frac{\alpha}{2}-r\right)_n = (-2-r)_n$ vanishes unless $r\geq n-2$. Using the definition of $(a)_r$ and $z\Gamma(z) = \Gamma(z+1)$ along with $\frac{(-m)_k}{k!} = (-1)^k \binom{m}{k}$, this sum reduces to
\[
\left < x^4 \widetilde{u}_n(1),\widetilde{u}_n(1) \right > = \sum_{k=n-2}^{n} (-1)^{n+k} \binom{k+2}{n}\binom{n}{k} (2+\nu+k)(1+\nu+k).
\]           
This sum is explicitly calculated as 
\[
\left < x^4 \widetilde{u}_n(1),\widetilde{u}_n(1) \right  > = 2 + 6 n (1 + n) + 3 \nu + 6 n \nu + \nu^2.
\]
Using the expression for $E_{n,2}$ in dimension $d\geq5$, we have
\[
E_{n,2} = 1 + 4n + 4n^2 + 2\nu + 4n\nu +\nu^2.
\]
But using $E_{n,0} = 2(2n+1+\nu)$, the relation $\frac{E_{n,2}}{2} = \frac{E_{n,0}^2}{8}$ holds exactly.

Hence when $d\geq5$ we have
\[
c_{n,1}=0, \quad c_{n,2}=0.
\]   
When $d=4$, we instead have 
\[
c_{n,1}=0,\quad c_{n,2} = -\frac{\mu}{2}\left(2 + 6 n (1 + n) + 3 \nu + 6 n \nu + \nu^2\right) \neq 0.
\]
\end{proof}

Note that we have now obtained the main theorem as stated in Section \ref{int}. 

\begin{table} [htb] 
\caption{Numerically computed real parts of quasinormal modes.} 
    \begin{tabular}{|l|l|l|l|l|l|}
        \hline
        $(\ell,n)$  & Asym. Exp.       & SLEIGN2       & WKB      & B--W method       \\ \hline
        $\ell = 3, n=0$ & 5.37639 & 5.91099 & 5.8668  & 5.8734  \\ 
        $\ell = 3, n=1$ & 6.45283 & 7.71884 & 7.6727  & 7.6776  \\ 
        $\ell = 3, n=2$ & 7.35226 & 9.47065 & 9.4189  & 9.4219  \\ 
        $\ell = 4, n=0$ & 6.46471 & 6.91806 & 6.8830  & 6.8889  \\ 
        $\ell = 4, n=1$ & 7.63913 & 8.75007 & 8.7139  & 8.7184  \\ 
        $\ell = 4, n=2$ & 8.63488 & 10.5348 & 10.4960 & 10.4996 \\ 
        $\ell = 5, n=0$ & 7.52937 & 8.00038 & 7.8945  & 7.8997  \\ 
        $\ell = 5, n=1$ & 8.78087 & 9.77257 & 9.7426  & 9.7466  \\ 
        $\ell = 5, n=2$ & 9.8549  & 11.5802 & 11.5482 & 11.5516 \\
        \hline
    \end{tabular} \label{table1}
\end{table}

\begin{figure} [htb] 
\includegraphics [width=\textwidth]{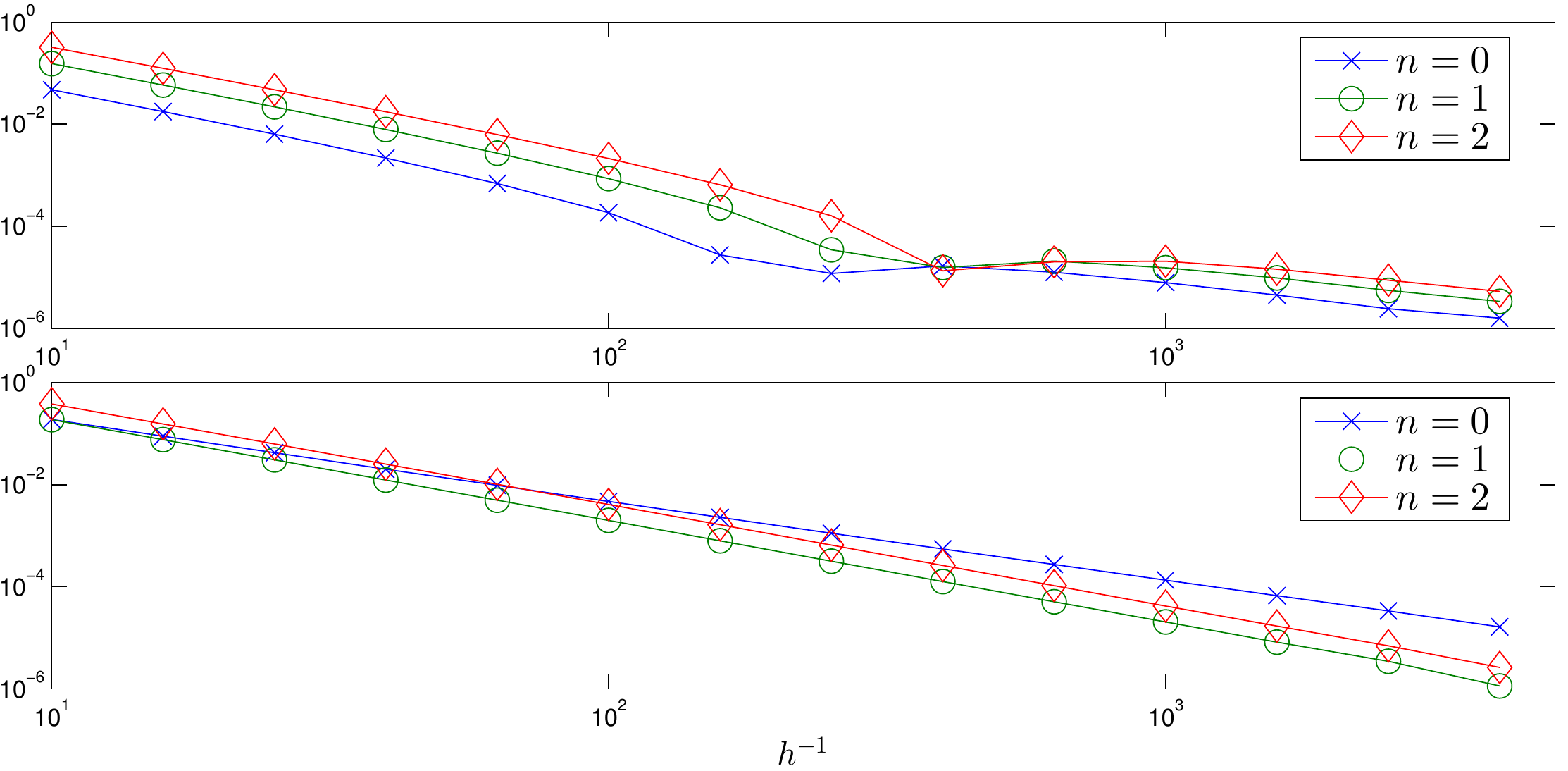}
\caption{A comparison between the asymptotic expansion for $E^\sharp_n(h)$ provided by Proposition \ref{thm:expansion} and $E^\sharp_n(h)$ as computed by SLEIGN2. Here the black hole parameters are $d=3, \mu = 1/10, \nu = 3/2$. Top: log-log plot of $h^{-1}$ against the difference between the SLEIGN2 value and the first two terms in the asymptotic expansion. Bottom: log-log plot of $h^{-1}$ against the difference between the SLEIGN2 value and the first three terms in the asymptotic expansion.} \label{f:4}
\end{figure}

\subsection{Numerical results} \label{numericalresults} In \cite{Festuccia:2008}, Festuccia and Liu derived a Bohr-Sommerfeld type quantization condition for resonances as $\ell\rightarrow\infty$ using WKB techniques. There have also been numerical studies in \cite{Berti:2009QNM} using what they term the ``Breit--Wigner resonance method.'' In Table \ref{table1} we compare our results with the two aforementioned results for the parameter values $d=3, \mu =1/10, \nu = 3/2$. The values in the table represent the real parts of resonances; the first column represents the three term expansion provided by Proposition \ref{thm:expansion}, namely
\[
\omega_n(h) \approx h^{-1}\left(1+2(2n+1+\nu)h+h^{3/2}E_{n,1}\right)^{1/2}.
\]
Here we computed the three terms, then took a square root, rather than using \eqref{eq:ellexpansion}.
In the second column we computed $\omega_n(h)$ using the program SLEIGN2 \cite{Bailey:2001}; this was done by considering the original equation in Sturm-Liouville form,
\[
-\DR\left(f\DR\psi\right)+ \left(\frac{(2\ell+d-2)^2-1}{4r^2} + \nu^2-\frac14 + \frac{\mu(d-1)^2}{4r^d} \right)\psi = \omega^2 f^{-1}\psi,
\]
and solving the eigenvalue problem on the interval $(r_{\mathrm{max}},\infty)$. The third column represents the Bohr-Sommerfeld approximation and the fourth column is the Breit--Wigner method; the latter two are taken from \cite{Berti:2009QNM}.

The real parts as computed by SLEIGN2 are in good agreement with the values in \cite{Berti:2009QNM}. Apart from the lowest mode, the asymptotic expansion did not reliably describe the real parts. However, this is only because $\ell$ is not large enough. In Figure~\ref{f:4} we compare the the real parts as computed by the expansion and SLEIGN2 for a larger range of values of $\ell$ and find the error behaves as predicted by Proposition~\ref{thm:expansion}.

\section*{Acknowledgements}
I would like to thank Maciej Zworski for suggesting the problem, a multitude of advice, and constant encouragement. I would also like to thank Jeffrey Galkowski and Semyon Dyatlov for valuable discussions and general helpfulness. Special thanks to Piotr Bizo\'n for pointing out the work of Festuccia--Liu \cite{Festuccia:2008}. I am grateful to UC Berkeley for support during Summer 2012 from the Graduate Division Summer Grant. This work was also partially supported by the NSF grant DMS-1201417. Finally, I am especially grateful to the two anonymous referees for their numerous comments and suggestions.

\end{document}